\documentclass[11pt,reqno]{amsart}
\pdfoutput=1

\usepackage{amsmath}
\usepackage{amsfonts}
\usepackage{amssymb}
\usepackage{enumerate}
\usepackage{amstext}
\usepackage{amsbsy}
\usepackage{amsopn}
\usepackage{bbm,amsthm}
\usepackage{amscd}
\usepackage[pdftex]{color}
\usepackage{amsxtra}
\usepackage{upref}
\usepackage{epstopdf}
\usepackage{graphicx,color}
\usepackage{hyperref}
\usepackage{longtable}
\usepackage{graphicx}
\usepackage[english]{babel}
\usepackage{soul}
\usepackage{array}
\usepackage{todonotes}

\DeclareFontFamily{OML}{rsfs}{\skewchar\font'177}
\DeclareFontShape{OML}{rsfs}{m}{n}{ <5> <6> rsfs5 <7> <8> <9> rsfs7
  <10> <10.95> <12> <14.4> <17.28> <20.74> <24.88> rsfs10 }{}
\DeclareMathAlphabet{\mathfs}{OML}{rsfs}{m}{n}

\newtheorem{theorem}{Theorem}
\newtheorem*{maintheorem}{Main Theorem}

\newtheorem{lemma}[theorem]{Lemma}
\newtheorem{proposition}[theorem]{Proposition}
\newtheorem{corollary}[theorem]{Corollary}

\theoremstyle{definition}

\theoremstyle{remark}
\newtheorem{remark}[theorem]{\bf Remark}

\numberwithin{equation}{section}
\numberwithin{theorem}{section}

\newcommand{\intav}[1]{\mathchoice {\mathop{\vrule width 6pt height 3 pt depth  -2.5pt
\kern -8pt \intop}\nolimits_{\kern -6pt#1}} {\mathop{\vrule width
5pt height 3  pt depth -2.6pt \kern -6pt \intop}\nolimits_{#1}}
{\mathop{\vrule width 5pt height 3 pt depth -2.6pt \kern -6pt
\intop}\nolimits_{#1}} {\mathop{\vrule width 5pt height 3 pt depth
-2.6pt \kern -6pt \intop}\nolimits_{#1}}}

\newcommand{\intavl}[1]{\mathchoice {\mathop{\vrule width 6pt height 3 pt depth  -2.5pt
\kern -8pt \intop}\limits_{\kern -6pt#1}} {\mathop{\vrule width 5pt
height 3  pt depth -2.6pt \kern -6pt \intop}\nolimits_{#1}}
{\mathop{\vrule width 5pt height 3 pt depth -2.6pt \kern -6pt
\intop}\nolimits_{#1}} {\mathop{\vrule width 5pt height 3 pt depth
-2.6pt \kern -6pt \intop}\nolimits_{#1}}}

\newcommand{\un}{\underline}

\newcommand{\ve}{\varepsilon}

\newcommand{\wh}{\widehat}
\newcommand{\vf}{\varphi}

\newcommand{\R}{\mathbb{R}}
\newcommand{\N}{\mathbb{N}}

\newcommand{\Z}{\mathbb{Z}}

\newcommand{\Hol}[1]{{\textrm{H\"ol}}_{#1}}

\renewcommand{\exp}[1]{{\rm exp}_{#1}}

\newcommand{\Sas}{d_{\rm Sas}}

\newcommand{\nuh}{{\rm NUH}}

\newcommand{\g}{{\overline g}}
\newcommand{\dd}{{\overline d}}
\newcommand{\mm}{M^*}
\newcommand{\reg}{M\backslash{\rm Sing}}
\newcommand{\sing}{{\rm Sing}}
\newcommand{\Sh}{{\widehat\Lambda}}

\newcommand{\vertiii}[1]{{\left\vert\kern-0.2ex\left\vert\kern-0.2ex\left\vert #1 
    \right\vert\kern-0.2ex\right\vert\kern-0.2ex\right\vert}}
\newcommand{\hatg}{{\widehat g}}

\title[Symbolic dynamics for non-uniformly hyperbolic flows]{Symbolic dynamics for \\ non-uniformly hyperbolic flows}
\author{J\'er\^ome Buzzi, Sylvain Crovisier, Yuri Lima, Chiyi Luo, and Dawei Yang}
\date{\today}

\begin{document}

\setcounter{tocdepth}{1}

\maketitle

\begin{abstract}
We construct symbolic dynamics for non-uniformly hyperbolic flows, in any dimension, possibly with fixed points.
More precisely, for each $\chi>0$, we code a set which has full measure for every $\chi$-hyperbolic invariant probability measure that gives zero mass to the set of fixed points. 
As a main application, we prove that a three dimensional $C^\infty$ flow with positive topological entropy on a closed manifold has finitely many ergodic measures of maximal entropy. 
For flows in any dimension, we also provide applications to the number of periodic orbits and to the Bernoulli property for equilibrium states of Hölder continuous potentials. 

The main technical result of this paper is a new method for handling singularities of vector fields by modifying the Riemannian metric. 
This technique is analogous to a blowup, which allows many results for nonsingular vector fields to be directly applied to vector fields with singularities.
\end{abstract}

\section{Introduction}

In 2013, Sarig constructed countable Markov partitions for $C^{1+\beta}$ ($\beta>0$) surface diffeomorphisms $f$ with positive topological entropy \cite{Sarig-JAMS}.
These partitions define finite-to-one symbolic extensions of $f$ by topological Markov shifts (TMS).
Representing a smooth dynamical system through a symbolic model allows its statistical analysis, including questions on the number and properties of measures of maximal entropy.
Such symbolic representations have played a central role in several recent developments, for instance in the proof by Buzzi, Crovisier and Sarig that every $C^\infty$ surface diffeomorphism with positive topological entropy is {\em strongly positively recurrent} (SPR) \cite{BCS-SPR}.
It is worth emphasizing that, in the setting of surface diffeomorphisms, positive topological entropy places us in the context of {\em non-uniform hyperbolicity}.

When studying the existence of symbolic representations for flows, a natural distinction arises according to whether or not the flow admits fixed points (also called singularities).
In the absence of fixed points -- equivalently, when the flow has positive speed, or when the vector field generating the flow has no zeroes -- the construction of countable Markov partitions has been successfully implemented for non-uniformly hyperbolic flows by Lima and Sarig \cite{Lima-Sarig}, Buzzi, Crovisier and Lima  \cite{BCL23} and Lima, Mongez and Nascimento \cite{LMN}. 
An ingredient used in this context is that flow boxes have uniform dynamical and geometrical properties.

In this paper, we construct Markov partitions for flows {\em with fixed points}, in any dimension.
Allowing the vector field to have zeroes introduces a major difficulty, since the aforementioned uniform dynamical and geometrical properties of flow boxes used in \cite{Lima-Sarig,BCL23,LMN} are no longer available.
Before describing our approach to this issue, let us state the main theorem.
We need some definitions.

Let $M$ be a $C^\infty$ closed Riemannian manifold, let $X$ be a $C^{1+\beta}$ vector field on $M$ with $\beta>0$, and let $\vf=\{\vf^t\}_{t\in\R}$ be the flow generated by $X$. 
Let $\chi>0$.
\medskip

\noindent
\label{def-chi-hyperbolic}
{\sc $\chi$-hyperbolic measure:} A $\vf$-invariant probability measure $\mu$ on $M$ is {\em $\chi$-hyperbolic} if $\mu$-a.e. point has well-defined Lyapunov exponents and if all those along directions transverse to $X$ lie outside of the interval $[-\chi,\chi]$.

\medskip
Let $\sing=\{x\in M:X(x)=0\}$ be the set of zeroes of $X$, which we call the {\em singular set}.
We are interested in coding $\chi$-hyperbolic measures that give measure zero to $\sing$.
The main result of this paper is the following theorem. 
See below for definitions.

\begin{maintheorem}\label{maintheorem}
Let $X$ be a $C^{1+\beta}$ vector field ($\beta>0$) on a $C^\infty$ closed Riemannian manifold $M$. 
For each $\chi>0$, there exist a locally compact topological Markov flow $(\Sigma_r,\sigma_r)$ and a map $\pi_r:\Sigma_r\to M$ which satisfy $\pi_r\circ \sigma_r^t=\vf^t\circ\pi_r$ for all $t\in\R$, such that:
\begin{enumerate}[{\rm (1)}]
\item The roof function $r$ is H\"older continuous and $\pi_r$ is $X$-scaled H\"older continuous.
\item $\pi_r[\Sigma_r^\#]$ has full measure for every $\chi$-hyperbolic measure on $M$ with $\mu(\sing)=0$.
\item $\pi_r$ is finite-to-one on $\Sigma_r^\#$, i.e. $\operatorname{Card}(\{z\in \Sigma_r^\#:\pi_r(z)=x\})<\infty$, for all $x\in \pi_r[\Sigma_r^\#]$.
\end{enumerate}
\end{maintheorem}

A more precise version of the Main Theorem is stated as Theorem~\ref{t.main.singular}. 
Observe that all $\chi$-hyperbolic measures that do not give mass to $\sing$ are coded. 
Under the assumption that $\mu(\sing)=0$, recurrence implies that the Lyapunov exponent in the flow direction is zero.

A topological Markov flow is the unit speed vertical flow on a suspension space whose basis is a topological Markov shift $\Sigma$ and whose roof function $r$ is continuous and uniformly bounded away from zero and infinity.
We endow $(\Sigma_r,\sigma_r)$ with a natural metric $d_r$, called the {\em Bowen-Walters metric}, that extends a natural metric $d$ on $\Sigma$ and that makes $\sigma_r$ a continuous flow. 
For this metric, $r$ and $\pi_r$ are H\"older continuous. 
Actually, the presence of fixed points improves the H\"older regularity of $\pi_r$, as we now explain. 
We say that $\pi_r$ is {\em $X$-scaled H\"older continuous} if there are $\kappa\in(0,1)$ and $C>0$ such that
$$ d(\pi_r(z),\pi_r(w))\leq C\|X(\pi_r(z))\|d_r(z,w)^\kappa $$
for every $z=(v,t),w=(u,s)\in\Sigma_r$ with $v,u$ are in the same cylinder set.

The set $\Sigma_r^\#$ is the {\em regular} set of $(\Sigma_r,\sigma_r)$, consisting of all elements of $\Sigma_r$ for which the symbolic coordinate has a symbol repeating infinitely often in the future and a symbol repeating infinitely often in the past. 
See Section \ref{s.notation} for the precise definitions.

\subsection{Scheme of proof}\label{section-scheme}
Our main theorem will follow from two constructions:
\begin{enumerate}[(1)]
    \item a change of metric making the flow nonsingular and noncompact but with some estimates which we call ``tame";
    \item a coding theorem applying to such tame settings.
\end{enumerate}

Let $g$ be the Riemannian metric on $M$, and let $\mm=\reg$. 
The first construction introduces a rescaled metric $\g$ on $\mm$ such that $(\mm,\g)$ is a $C^\infty$ complete Riemannian manifold with controlled geometry and the restriction of $X$ to $\mm$ is a {\em non-singular} flow with uniformly bounded $C^{1+\beta}$ norm.
The metric is defined by 
$$
\g=\frac{g}{\ve^2 f^2}
$$
where $f:\mm\to (0,+\infty)$ is a carefully chosen positive and bounded $C^\infty$ function, comparable to $\|X\|$, and $\ve>0$ is small. 
For the precise statement, we need two definitions.

\medskip
\label{def-tame-manifold}
\noindent
{\sc Tame manifold:}
A $C^\infty$ Riemannian manifold $(N,g_{N})$ is called {\em tame} if:
\begin{enumerate}[{\rm (1)}]
\item it is complete;
\item $\|g_N\|_{C^2}<\infty$;
\item its injectivity radius is positive.
\end{enumerate}
In practice, we verify condition (2) by estimating derivatives of the coefficients of $g_N$.
\medskip

\noindent
{\sc Tame vector field:} A $C^{1+\beta}$ vector field $X$ on a $C^\infty$ Riemannian manifold $(N,g_{N})$ is {\em tame} if:
\begin{enumerate}[{\rm (1)}]
\item $0<\inf_{x\in N} \|X(x)\|\leq \sup_{x\in N}\|X(x)\|<\infty$;
\item $DX$ is uniformly bounded;
\item there are constants $r,C>0$ such that $DX$ has $\beta$-H\"older norm bounded by $C$ inside any ball of radius $r$.   
\end{enumerate}

Now we are ready to state the theorems performing the two constructions. Denote the restriction of $X$ to $\mm$ also by $X$.

\begin{theorem}\label{theorem-metric}
Let $X$ be a $C^{1+\beta}$ vector field ($\beta>0$) on a $C^\infty$ closed Riemannian manifold $(M,g)$, and let $\mm=M\setminus{\rm Sing}$. 
There is $f:\mm\to(0,+\infty)$ a $C^\infty$ bounded function with 
$$
\frac{1}{2}\leq \frac{f(x)}{\|X(x)\|}\leq 2,\ \text{ for all }x\in\mm,
$$
such that for all $\ve>0$ small enough the metric $\g=g/(\ve f)^2$ satisfies:
\begin{enumerate}[{\rm (1)}]
\item $(\mm,\g)$ is a {\em tame} $C^\infty$ Riemannian manifold and
there is $C>0$ such that
$$ d(x,y)\leq C\|X(x)\|\dd(x,y),\ \text{ for all $x,y\in\mm$ s.t. $d(x,y)<C^{-1}\|X(x)\|$}. $$
\item $X$ is a {\em tame} $C^{1+\beta}$ vector field on $(\mm,\g)$.
\end{enumerate}
\end{theorem}

Above, $\dd$ is the distance induced by $\g$.
The proof of Theorem \ref{theorem-metric} is in Section \ref{section-theorem-metric}.

Our second construction generalizes the symbolic dynamics of  \cite{BCL23,LMN} to this possibly noncompact setting. 
We have the following general coding theorem for tame manifolds and flows.Recall that $\chi$-hyperbolicity  has been defined on page~\pageref{def-chi-hyperbolic}.

\begin{theorem}\label{theorem-coding}
Let $X$ be a {\em tame} $C^{1+\beta}$ vector field ($\beta>0$) on a {\em tame} $C^\infty$ Riemannian manifold $(\mm,\g)$.
For each $\chi>0$, there exist a locally compact topological Markov flow $(\Sigma_r,\sigma_r)$ and a map $\pi_r:\Sigma_r\to \mm$ such that $\pi_r\circ \sigma_r^t=\vf^t\circ\pi_r$ for all $t\in\R$, satisfying:
\begin{enumerate}[{\rm (1)}]
\item The roof function $r$ and the projection $\pi_r$ are H\"older continuous.
\item $\pi_r[\Sigma_r^\#]$ has full measure for every $\chi$-hyperbolic probability measure on $\mm$.
\item $\pi_r$ is finite-to-one on $\Sigma_r^\#$, i.e. $\operatorname{Card}(\{z\in \Sigma_r^\#:\pi_r(z)=x\})<\infty$, for all $x\in \pi_r[\Sigma_r^\#]$.
\end{enumerate}
\end{theorem}

The proof of Theorem \ref{theorem-coding} is given in Section \ref{section-theorem-coding}. 
Its general strategy is the same as that of \cite{BCL23,LMN}. We will focus on detailing all the changes. 
They are related to the non-compactness of $\mm$: while all {\em local} arguments of \cite{BCL23,LMN} work, the {\em global} arguments need to be justified and/or adapted to {\em semi-local} statements (for instance, the Poincaré section we choose will only be {\em locally finite}). 

\begin{remark}\label{rmk-exponents}
To define Lyapunov exponents one uses the metric and applies the Oseledets theorem, which requires some integrability assumptions. We claim that Lyapunov exponents do exist in the tame setting of Theorem~\ref{theorem-coding} and are preserved by the change  of the metric in Theorem~\ref{theorem-metric}.

Indeed, tame vector fields give rise to bounded cocycles (defined by the variational ODE) since both $X$ and $DX$ are uniformly bounded. 
Hence the integrability condition of the Oseledets theorem is satisfied for any tame vector field and therefore all invariant probability measures have well-defined Lyapunov exponents in the tame setting of Theorem~\ref{theorem-coding}.

Poincaré recurrence guarantees that a measurable change of metric cannot modify the values of the Lyapunov exponents as long as the exponents exist for both the initial and the modified metrics. 
Hence, any invariant measure $\mu$ on $M\setminus\sing$ has the same Lyapunov exponents with respect to both metrics $g$ and $\overline g$ of Theorem~\ref{theorem-metric}.
\end{remark}

During the preparation of this manuscript, we learned that Li and Liu proved a theorem essentially equivalent to our Main Theorem, see \cite[Theorem A]{Li-Liu}. 
\footnote{The fifth author of this paper gave a public talk of the Main Theorem on March 2026; see the information available at \url{http://tianyuanmc.jlu.edu.cn/info/1021/3335.htm}.} 
Both proofs use the general scheme of \cite{BCL23,LMN}. 
There is, though, a crucial distinction on how this is implemented. 
Li and Liu work directly with the original metric, and adapt the arguments of \cite{BCL23,LMN} to flow boxes with sizes comparable to $\|X\|$. 
After a scaling, they recover uniform estimates of these and related objects. 
This approach was first introduced by Liao in a series of papers \cite{Liao1979,Liao1985,Liao1989}, and constitutes an effective tool when dealing with singular flows. 
The modern presentation and its geometric meaning were first given by Gan and Yang \cite{Gan-Yang-2018}.
Li and Liu then apply the approach of \cite{BCL23,LMN} to the rescaled objects.

Our proof uses rescaling in a {\em different way}: instead of rescaling the objects separately, we rescale the Riemannian metric.
The vector field becomes nonsingular, whilst the manifold is no longer compact: this setting is much closer to the context of \cite{BCL23,LMN}, and provides an efficient proof of the Main Theorem. 
Additionally, we believe that the generality of Theorem \ref{theorem-metric} will allow its application to other questions about singular flows.

\subsection{Applications}

We provide four applications of the Main Theorem:
\begin{enumerate}[$\circ$]
\item A lower exponential bound on the number of closed orbits.
\item The construction of irreducible coding of homoclinic classes of measures.
\item The Bernoulli property for equilibrium states of Hölder continuous potentials.
\item The finiteness of the set of ergodic measures of maximal entropy in dimension three.
\end{enumerate}
These applications have also been obtained by Li and Liu \cite{Li-Liu}.
Each of them requires additional hypotheses on the flow, as we now describe.

A {\em simple closed orbit} of length $\ell$ is a length-parametrized curve $\gamma(t)=\vf^t(x)$, $0\leq t\leq\ell$, such that $\gamma(0)=\gamma(\ell)$ and $\gamma(0)\neq \gamma(t)$ for $0<t<\ell$. Defining the {\em trace} of $\gamma$ to be the set $\{\gamma(t):0\leq t\leq\ell\}$, we let $[\gamma]$ denote the equivalence class of the relation $\gamma_1\sim\gamma_2\Leftrightarrow\gamma_1,\gamma_2$ have equal lengths and traces. Let ${\rm Per}_T(\vf):=\#\{[\gamma]:\gamma\textrm{ is simple and }\ell(\gamma)\leq T\}$, and ${\rm Per}_{T,\chi}(\vf)$ be the number of such periodic orbits that are $\chi$-hyperbolic (the Lyapunov exponents transverse to the flow direction are outside $[-\chi,\chi]$).

\begin{theorem}\label{Thm-Flows-periodic}
Let $X$ be a $C^{1+\beta}$ vector field ($\beta>0$) on a $C^\infty$ closed Riemannian manifold $M$. 
If $\vf$ has a hyperbolic measure of maximal entropy, then $\exists C>0$ such that ${\rm Per}_T(\varphi)\geq C{e^{hT}}/{T}$ for all $T$ large enough.
More precisely, if $\vf$ has a $\chi$-hyperbolic measure of maximal entropy,
then $\exists C>0$ such that ${\rm Per}_{T,\chi}(\varphi)\geq C{e^{hT}}/{T}$ for all $T$ large enough.
\end{theorem}

\begin{proof}
Theorem \ref{Thm-Flows-periodic} is deduced from the Main Theorem above exactly as in \cite[Theorem 8.1]{Lima-Sarig}.
\end{proof}

In applications, it is useful to work with {\em irreducible} Markov shifts since, among other properties, they are topologically transitive and  carry at most one equilibrium state for each H\"older continuous potential (see Section~\ref{s.notation}).
This is related to the notion of homoclinically related measures and of \emph{homoclinic classes of measures} (see Section~\ref{sec.homoclinic}).
We prove the following theorem.

\begin{theorem}\label{thm.homoclinic}
In the setting of the Main Theorem, let $\mu$ be an ergodic hyperbolic measure with $\mu(\sing)=0$.
Then $\Sigma_{r}$ contains an irreducible component ${\Sigma'_{r}}$ with the following properties:
\begin{enumerate}[{\rm (1)}]
\item $\Sigma_r'$ lifts all ergodic $\chi$-hyperbolic measures $\nu$ with $\nu(\sing)=0$ that are homoclinically related to $\mu$.
\item If $\overline{\nu}$ is a $\sigma_r$-invariant probability measure, then its projection $\nu=(\pi_r)_*{\overline\nu}$
is homoclinically related to $\mu$. 
\end{enumerate}
\end{theorem}

\begin{proof}
Exactly as in \cite[Theorem 1.1]{LMN}.
\end{proof}

This implies the following result for equilibrium states. 
Call a potential $\psi:M\to\R$ {\em admissible} if $\psi\circ\pi_r:\Sigma_r\to\R$ is H\"older continuous, for $\pi_r$ given by the Main Theorem. 
Clearly, every H\"older continuous $\psi$ is admissible.

\begin{corollary}\label{cor.local-uniq}
In the setting of the Main Theorem, let $\mu$ be an ergodic hyperbolic measure, and let $\psi:M\to \mathbb{R}$ be an admissible potential.
Then there is at most one hyperbolic measure $\nu$ with $\nu(\sing)=0$ which is homoclinically related to $\mu$ and satisfies 
$$    h(\vf,\nu)+\int \psi  d\nu
    =\sup\left\{h(\vf,\eta) +\int \psi  d\eta:\begin{array}{l}\eta \text{ \em ergodic, hyperbolic with $\eta(\sing)=0$,}\\
\text{\em  and homoclinically related to $\mu$}\end{array}\right\}.
$$
When it exists, this measure is Bernoulli up to a period.
\end{corollary}

\begin{proof}
The uniqueness follows as in \cite[Corollary 1.2]{LMN}.
The Bernoulli property is obtained as in \cite{LLS-2016}, where the authors consider an irreducible topological Markov flow $(\Sigma_r,\sigma_r)$ and an equilibrium state of a Hölder continuous potential and analyze its Pinsker factor:
\begin{enumerate}[$\circ$]
\item If $r$ is {\em arithmetic}, then the Pinsker factor
is isomorphic to a rotation \cite[Theorem 4.7]{LLS-2016},
and the measure is Bernoulli times a rotation
\cite[Lemma 4.8]{LLS-2016}.
\item If $r$ is {\em not arithmetic}, then the Pinsker factor is trivial
and the measure is Bernoulli \cite[Theorem 5.1]{LLS-2016}.
\end{enumerate}
This concludes the proof of the corollary.
\end{proof}

The above corollary reduces the question of the number of equilibrium states to the study of the following two questions: (a) the eventual non-hyperbolic equilibrium states, perhaps proving that they do not exist; (b) the number of homoclinic classes that may carry an equilibrium state, perhaps proving that only one can have maximal pressure. For possibly singular, three dimensional  flows and potential $\psi\equiv0$, the Ruelle inequality shows that all measures maximizing the entropy are hyperbolic, solving (a); a variant of the strong positive recurrence introduced in \cite{BCS-SPR} for diffeomorphisms yields the finiteness in (b).
We thus prove the following result.

\begin{theorem}\label{Thm:finiteness-C-infty}
Let $M$ be a $C^\infty$ closed three dimensional Riemannian manifold, and let $X$ be a $C^\infty$ vector field on $M$.
If $\vf$ has positive topological entropy, then $\vf$ has finitely many ergodic measures of maximal entropy.
\end{theorem}

The analogous result for nonsingular three dimensional flows was proved recently by Zang \cite{Yuntao2025}, and Burguet, Luo and Yang proved the finiteness of the set of equilibrium states for a large class of equilibrium states \cite{BLY-ESPR}. 

The proof of Theorem \ref{Thm:finiteness-C-infty} is in Section \ref{section-finiteness}. 
It follows  a  strategy similar to \cite{BCS-SPR}, proving a version of the strong recurrence property for flows.  

\subsection{Preliminaries}\label{s.notation}

We fix a $C^\infty$ closed Riemannian manifold $M$ of dimension $n$, and let $X:M\to TM$ be a $C^{1+\beta}$ vector field ($\beta>0$) and $\varphi=(\varphi^t)_{t\in \R}$ be the flow generated by $X$.
We will denote the value of the vector field $X$ at $x$ by $X(x)$.
Given a set $Y\subset M$ and an interval $I\subset\R$, write $\vf^I(Y):=\bigcup_{t\in I}\vf^t(Y)$. 

The Riemannian metric on $M$ induces a Riemannian metric $\Sas(\cdot,\cdot)$ on $TM$, called the {\em Sasaki metric} (see for instance~\cite{Burns-Masur-wilkinson}). 
By the regularity of $X$, there is $L>0$ such that
$$
d_{\rm Sas}(X(x),X(y))\leq L d(x,y),\ \text{ for all }x,y\in M.
$$
In the sequel, we fix finitely many charts and write the above inequality as 
$\|X(x)-X(y)\|\leq L d(x,y)$ when $x,y$ are in the image of a common chart.

Let $\mathfs G=(V,E)$ be an oriented graph, where $V,E$ are the vertex and edge sets.
We denote edges by $v\to w$, and assume that $V$ is countable.

\medskip
\noindent
{\sc Topological Markov shift (TMS):} It is a pair $(\Sigma,\sigma)$
where
$$
\Sigma:=\{\text{$\Z$-indexed paths on $\mathfs G$}\}=
\left\{\un{v}=\{v_n\}_{n\in\Z}\in V^{\Z}:v_n\to v_{n+1}, \forall n\in\Z\right\}
$$
is the symbolic space and $\sigma:\Sigma\to\Sigma$, $[\sigma(\un v)]_n=v_{n+1}$, is the {\em left shift}. 
We endow $\Sigma$ with the distance $d(\un v,\un w):={\rm exp}[-\inf\{|n|\in\Z:v_n\neq w_n\}]$.
The {\em regular set} of $\Sigma$ is
$$
\Sigma^\#:=\left\{\un v\in\Sigma:\exists v,w\in V\text{ s.t. }\begin{array}{l}v_n=v\text{ for infinitely many }n>0\\
v_n=w\text{ for infinitely many }n<0
\end{array}\right\}.
$$

\medskip
We only consider TMS that are \emph{locally finite}, i.e.
for all $v\in V$ the number of ingoing edges $u\to v$ and outgoing edges $v\to w$ is finite.

\medskip
Given $(\Sigma,\sigma)$ a TMS, let $r:\Sigma\to(0,+\infty)$ be continuous.
For $n\geq 0$, let
$r_n=r+r\circ\sigma+\cdots+r\circ \sigma^{n-1}$ be $n$-th {\em Birkhoff sum} of $r$,
and extend this definition for $n<0$
in the unique way such that the {\em cocycle identity} holds: $r_{m+n}=r_m+r_n\circ\sigma^m$, $\forall m,n\in\Z$.

\medskip
\noindent
{\sc Topological Markov flow (TMF):} The TMF defined
by $(\Sigma,\sigma)$ and the \emph{roof function} $r$ is the pair $(\Sigma_r,\sigma_r)$ where
$\Sigma_r:=\{(\un v,t):\un v\in\Sigma, 0\leq t<r(\un v)\}$
and $\sigma_r:\Sigma_r\to\Sigma_r$ is the flow on $\Sigma_r$ given by
$\sigma_r^t(\un v,t')=(\sigma^n(\un v),t'+t-r_n(\un v))$, where
$n$ is the unique integer such that $r_n(\un v)\leq t'+t<r_{n+1}(\un v)$.
We endow $\Sigma_r$ with a natural metric $d_r(\cdot,\cdot)$,
called the {\em Bowen-Walters metric}, such that $\sigma_r$ is a continuous flow \cite{Bowen-Walters-Metric,Lima-Sarig}.
The {\em regular set} of $(\Sigma_r,\sigma_r)$ is $\Sigma_r^\#=\{(\un v,t)\in\Sigma_r:\un v\in \Sigma^\#\}$.

\medskip
In other words, $\sigma_r$ is the unit speed vertical flow on $\Sigma_r$ with the identification
$(\un v,r(\un v))\sim (\sigma(\un v),0)$. 
The roof functions we will consider are H\"older continuous.
In this case, there exist $\kappa,C>0$ such that $d_r(\sigma_r^t(z),\sigma_r^{t}(z'))\leq C d_r(z,z')^\kappa$
for all $|t|\leq 1$ and $z,z'\in\Sigma_r$, see \cite[Lemma 5.8]{Lima-Sarig}.

\medskip
\noindent
{\sc Irreducible component:}
If $\Sigma$ is a countable Markov shift defined by an oriented graph
$\mathfs{G}=(V,E)$, its \emph{irreducible components} are the subshifts $\Sigma'\subset \Sigma$ over
maximal subsets $V'\subset V$ satisfying the following condition:
$$\forall v,w\in V',\;\exists \un v\in \Sigma \text{ and } n\geq 1\text{ such that } v_0=v \text{ and } v_n=w.$$
Each irreducible component $\Sigma'$ of $\Sigma$ defines an irreducible component of $\Sigma_r$, 
equal to the set of
elements $(\un v,t)\in \Sigma_r$ with $\un v\in\Sigma'$.

\medskip
Buzzi and Sarig proved, for an irreducible TMS, that a large class of Hölder continuous potentials have at most one equilibrium state \cite{Buzzi-Sarig}.

\section{Rescaled metric: proof of Theorem \ref{theorem-metric}}\label{section-theorem-metric}

In this section we prove Theorem \ref{theorem-metric}. Recall that:
\begin{enumerate}[$\circ$]
\item $(M,g)$ is a $C^\infty$ closed Riemannian manifold.
\item $X$ is a $C^{1+\beta}$ vector field ($\beta>0$) on $M$.
\item $\sing=\{x\in M:X(x)=0\}$ is the singular set of $X$, and $\mm=M\setminus\sing$.
\end{enumerate}
We begin constructing a scaling function $f$.

For $k\in \N$, we say that a map $f: \R^n \to \R^1$ is $C^k$ if for every 
multi-index $L:=(\ell_1,\ldots,\ell_n)\in (\N\cup \{0\})^n$ with $|L|:=\ell_1+\cdots+\ell_n\le k$, the partial derivative
$\partial_{L}f=\tfrac{\partial^{|L|}f}{\partial^{\ell_1}x_1 \cdots \partial^{\ell_n}x_n}$ exists and is continuous.
For each $x\in \R^n$, we denote by $ \|(D^kf)_x\|$ the maximum of the absolute value of all $k$-th order partial derivatives of $f$ at $x$, and
we say that $f$ is $C^{\infty}$ if $f$ is $C^k$ for any $k\in \N$.

Let $f,g: \R^n \to \R$ be two $C^k$ functions.
For multi-indices $L,\widetilde L$ as above, we write $\widetilde L\leq L$ if $\widetilde\ell_i\leq \ell_i$ for all $1\leq i\leq n$.
In this case, $L-\widetilde L=(\ell_i-\widetilde\ell_i)_{i=1}^n$ is also a multi-index, and we write $\binom{L}{\widetilde L}=\prod_{i=1}^{n} \tbinom{\ell_i}{\widetilde\ell_i}$. In the following, we will use the \textit{general Leibniz rule}: if
$L$ is a multi-index with $|L|\leq k$ then
\begin{equation}\label{eq:Leibniz-rule}
	[\partial_{L}(fg)]_x=\sum_{\widetilde L \leq L} \binom{L}{\widetilde L} (\partial_{\widetilde L} f)_x \cdot (\partial_{L-\widetilde L} g)_x.
\end{equation}

\begin{proposition}\label{lemma-f}
There is a $C^{\infty}$ function $f:\mm\to (0,+\infty)$ such that:
\begin{enumerate}[{\rm (1)}]
	\item $\tfrac{1}{2} \leq f(x)/\|X(x)\| \leq 2$ for all $x\in \mm$.
	\item For each $k\geq 1$, $\exists C_k>1$ such that $f(x)^{k-1} \|(D^kf)_x\|\leq C_k$ for all $x\in \mm$.
\end{enumerate}
\end{proposition}

\begin{proof}
Recall that we work in selected charts where $\|X(x)-X(y)\|\leq L d(x,y)$ for $x,y\in M$ in the image of a common chart. 
For $x\in\mm$, let $\mathfrak r(x)=\frac{1}{20L}\|X(x)\|$. 
Then
\begin{eqnarray}\label{slow-variation}
\frac{1}{2}\leq \frac{\|X(x)\|}{\|X(y)\|}\leq 2, \ \text{for all}~x,y\in \mm~\text{with}~B(x,4\mathfrak{r}(x))\cap B(y,4\mathfrak{r}(y))\neq \emptyset.
\end{eqnarray}
By the Besicovitch  covering lemma, the family $\{B(x,\mathfrak r(x)):x\in\mm\}$ contains a countable subcover $\mathcal U=\{U_i=B(x_i,\mathfrak r(x_i)):i\geq 1\}$ with the following local finiteness property:
there exists $B>0$ such that for every $i\geq 1$ it holds 
\begin{equation}\label{local-finite}
\#\{j\geq 1: B(x_j,4\mathfrak r(x_j))\cap B(x_i,4\mathfrak r(x_i))\neq \emptyset \}<B.
\end{equation}
Consider a $C^{\infty}$ bump function $\theta:\mathbb R^n\to[0,1]$ such that:
\begin{enumerate}[$\circ$]
\item $\theta(v)=1$ for $\|v\|\leq 1$;
\item $\theta(v)=0$ for $\|v\|\geq 2$;
\item $0<\theta(v)<1$ for $1<\|v\|<2$.
\end{enumerate}
For each $k\geq 0$, let $H_k=\sup_{v \in \R^n} \|D^k \theta_v\|<\infty$.
For simplicity, in the sequel we identify $\mm$ with $\mathbb R^n$.
For each $j\ge 1$, we consider the rescaled bump function $\theta_j:B(x_j,4\mathfrak{r}(x_j)) \to [0,1]$ 
defined by 
$$
\theta_j(v)= \theta\left(\frac{v-x_j}{\mathfrak{r}(x_j)}\right),
$$
so that $\theta_j\equiv 1$ on $B(x_j, \mathfrak{r}(x_j))$ and
${\rm supp}(\theta_j)\subset B(x_j, 2\mathfrak{r}(x_j))$.

Now define $s:\mm\to \R$ by $s(x)=\sum_{j \geq 1} \theta_j(x)$. We note the following:
\begin{enumerate}[$\circ$]
\item $\mathcal{U}$ covers $\mm$, hence $s(x)\geq 1$;
\item By property (\ref{local-finite}), $s$ is $C^{\infty}$ and $s(x)\le B$.
\end{enumerate}
Finally, let $f_j:\mm\to[0,1]$ by $f_j(x):=\frac{\theta_j(x)}{s(x)}$.
Then $\sum_{j\ge 1} f_j\equiv 1$ and
\begin{equation}\label{estimate-derivative-bump}
\left\{
\begin{array}{l}
\|(D^k\theta_j)_x\|\leq 2^k H_k \cdot \mathfrak{r}(x)^{-k}\\
\hspace{.15cm} \|(D^k s)_x\|\leq 2^k H_k B \cdot \mathfrak{r}(x)^{-k}
\end{array}\right.
,\forall x\in \mm, j\geq 1.
\end{equation}
Using these, we now estimate the norms of $D^k(s^{-1})$ and $D^k f_j$.
We first define inductively $(E_k)_{k\geq 0}$ such that $\|[D^k (s^{-1})]_x\|\leq E_k \cdot \mathfrak{r}(x)^{-k}$
for all $x\in\mm$. Take $E_0=1$. By (\ref{estimate-derivative-bump}),
$$
\|[D(s^{-1})]_x\|=\frac{\|Ds_x\|}{s(x)^2}\leq \|Ds_x\|\leq \underbrace{2H_1B}_{=:E_1} \cdot \mathfrak{r}(x)^{-1}
=E_1\cdot \mathfrak{r}(x)^{-1}.$$
Assume  $E_0,E_1,\ldots,E_{k-1}$ satisfying the above inductive hypothesis.
By the general Leibniz rule \eqref{eq:Leibniz-rule}, for each multi-index $L$
with $|L|=k$ it holds
$0=[\partial_{L}(s\cdot s^{-1})]_x=\sum_{\widetilde L \leq L} \binom{L}{\widetilde L} [\partial_{\widetilde L} s]_x \cdot [\partial_{L-\widetilde L} (s^{-1})]_x$.
Hence
\begin{align*}
	&\| [\partial_{L} (s^{-1})]_x \|
	\leq s(x)^{-1}\sum_{\widetilde L \leq L\atop{|\widetilde L|\geq 1}} \binom{L}{\widetilde L} \| (\partial_{\widetilde L} s)_x\| \cdot \|[\partial_{L-\widetilde L} (s^{-1})]_x\| \\
	&\leq \sum_{\ell=1}^{k} \sum_{\widetilde L \leq L\atop{|\widetilde L|=\ell}} \binom{L}{\widetilde L} \|(D^{\ell}s)_x\| \cdot \|[D^{k-\ell} (s^{-1})]_x\| 
    =
    \sum_{\ell=1}^{k} \binom{k}{\ell} \|(D^{\ell}s)_x\| \cdot \|[D^{k-\ell} (s^{-1})]_x\|.
\end{align*}
By (\ref{estimate-derivative-bump}) and the induction hypothesis, we get that
\begin{align*}
	\| [D^{k} (s^{-1})]_x\|
	&=\max_{|L|=k} \|[\partial_{L} (s^{-1})]_x\|
	\leq \sum_{\ell=1}^{k} \binom{k}{\ell} \|(D^{\ell}s)_x\| \cdot \|[D^{k-\ell} (s^{-1})]_x\| \\
	&\leq  \underbrace{\left[\sum_{\ell=1}^{k} \binom{k}{\ell}\cdot 2^\ell H_\ell B\cdot E_{k-\ell}\right]}_{=:E_k}\mathfrak{r}(x)^{-k}=:E_k\cdot\mathfrak{r}(x)^{-k},
\end{align*}
and the induction is complete.

We now define $(F_k)_{k\geq 0}$ such that, for all $k\ge0$
\begin{equation}
\|(D^k f_j)_x\|\leq F_k \cdot \mathfrak{r}(x)^{-k}, \ \ \forall x\in\mm, j\geq 1.
\end{equation}
One can take $F_0=1$. Differentiating $f_j=\theta_j\cdot s^{-1}$ gives
$$
	\|(D^k f_j)_x\| \leq \sum_{\ell=0}^{k} \binom{k}{\ell} \|(D^{\ell}\theta_j)_x\| \cdot \|[D^{k-\ell} (s^{-1})]_x\|
	               	\leq  \underbrace{\left[\sum_{\ell=0}^{k} \binom{k}{\ell} \cdot 2^\ell H_\ell B\cdot E_{k-\ell}\right]}_{=:F_k}\mathfrak{r}(x)^{-k},
$$
which proves the assertion.

With these estimates at hand, we define $f:\mm\to(0,+\infty)$ by
$$
f(x)=\sum_{i\geq 1} f_i(x) \|X(x_i)\|.
$$
We check properties (1) and (2). Let $x\in \mm$. If $x\in{\rm supp}(f_i)\subset B(x_i,2\mathfrak r(x_i))$
then (\ref{slow-variation}) implies $\tfrac{1}{2}\leq \tfrac{\|X(x_i)\|}{\|X(x)\|}\leq 2$.
Summing up on $i$ gives (1). We now prove (2). Define $C_k=2^kBF_k(20L)^k$. 
By property (1) and the Besicovitch bound (\ref{local-finite}), 
\begin{align*}
	&[f(x)]^{k-1} \|(D^kf)_x\|
	\leq  [f(x)]^{k-1} \sum_{i\geq 1} \|(D^k f_i)_x\| \cdot  \|X(x_i)\| \\
&\leq 2^{k-1}  \|X(x)\|^{k-1} \sum_{i\geq 1} \|(D^k f_i)_x\| \cdot  2\|X(x)\|=
2^{k}  \|X(x)\|^{k} \sum_{i\geq 1} \|(D^k f_i)_x\|\\
&\leq 2^{k}  \|X(x)\|^{k} \cdot B F_k \mathfrak r(x)^{-k}=C_k,
\end{align*}
as claimed. This completes the proof of the proposition.
\end{proof}

Now we introduce a one-parameter family of metrics on $\mm$.

\medskip
\noindent
{\sc Metrics $\g^\ve$:} For each $\ve>0$, define the metric
$$
\g^\ve=\frac{g}{\ve^2 f^2}\cdot
$$
For small $\ve>0$, $\g^\ve\ge g$ and therefore it defines smaller balls. We let $\vertiii{\cdot}_\ve$ and $\dd_\ve$ denote the norm and metric induced by $\g^\ve$.

\medskip
The idea in the above definition is that $\g^\ve$ ``zooms in'' the manifold $M$ close to $\sing$. We will prove that, as $\ve\to0$, the metric becomes close to a constant metric (small derivatives). Moreover, the first derivative of $X$ does not change much, while the variation of its higher derivatives improve. In the sequel, we quantify this idea and prove other general properties on $\g^\ve$.

\begin{lemma}\label{lemma-complete}
The following holds for all $\ve>0$:
\begin{enumerate}[{\rm (1)}]
\item $(\mm,\g^\ve)$ is a $C^\infty$ complete Riemannian manifold.
\item For every \(x\in \mm\), it holds that
$$\frac{1}{8\ve f(x)} d(y,z)\leq \dd_\ve(y,z)\leq \frac{8}{\ve f(x)}d(y,z),\ \text{ for all }y,z\in B(x,\mathfrak r(x)).$$
\end{enumerate}
($B(x,\mathfrak r(x))$ is the larger ball with respect to the initial metric $g$.)
\end{lemma}

\begin{proof}
(1) Fix $\ve>0$. It is clear that $\g^\ve$ is a $C^\infty$ Riemannian metric, hence we just need to prove completeness. Let $\gamma:[0,1]\to M$ be a curve such that $\gamma(t)\in\mm$ for all $t\in [0,1)$ and $\gamma(1)\in \sing$. We claim that the restriction of $\gamma$ to $[0,1)$ has infinite length in the metric $\g^\ve$. For that, for each $n\geq 1$ let
$[t_n,s_n)$ be a subinterval such that $\{d(\gamma(t),\sing):t_n\le t< s_n\}=(e^{-n-1},e^{-n}]$.
This defines disjoint intervals of $[0,1]$. Letting $\gamma_n$ be the restriction of $\gamma$ to $[t_n,s_n]$, we have
$\ell(\gamma)\geq \sum_{n\geq 1}\ell(\gamma_n)$ both in the metrics $g$ and $\g^\ve$. Letting $y\in\sing$ be such that $d(x,\sing)=d(x,y)$, then
$$
f(x)\leq 2\|X(x)\|=2\|X(x)-X(y)\|\leq 2L d(x,y)=2L d(x,\sing)
$$
and so for $t\in [t_n,s_n)$ it holds
$$
\vertiii{\gamma'(t)}_\ve= \frac{\|\gamma'(t)\|}{\ve f(\gamma(t))}
\geq \frac{\|\gamma'(t)\|}{2\ve L d(\gamma(t),\sing)}
\geq (2\ve L)^{-1} e^{n}\|\gamma'(t)\|.
$$
Integrating, we conclude that the length of $\gamma_n$ in the metric $\g^\ve$
is
\begin{align*}
&\ \ell_{\g^\ve}(\gamma_n)=\int_{t_n}^{s_n}\vertiii{\gamma'(t)}_{\ve}dt
\geq (2\ve L)^{-1} e^{n}\int_{t_n}^{s_n}\|\gamma'(t)\|dt=
(2\ve L)^{-1} e^{n}\ell_g(\gamma_n)\\
&\geq 
(2\ve L)^{-1} e^{n}(e^{-n}-e^{-n-1})=(2\ve L)^{-1}(1-e^{-1}).
\end{align*}
Since this latter expression does not depend on $n$, we conclude that $\ell_{\g^\ve}(\gamma)=+\infty$. 

\medskip
\noindent
(2) Let $y\in B(x,\mathfrak r(x))$. By Proposition \ref{lemma-f} and estimate (\ref{slow-variation}),
we have $\tfrac{1}{8} \leq \tfrac{f(x)}{f(y)}\leq 8$, and so
\begin{equation}\label{estimate-eps-norm}
\frac{1}{8\ve f(x)}\|v\|\leq \vertiii{v}_{\ve}\leq \frac{8}{\ve f(x)}\|v\|,
\ \ \forall v\in T_yM.
\end{equation}
Integrating, the assertion follows.
\end{proof}

The next proposition estimates the derivatives of $\g^\ve$.

\begin{proposition}\label{prop-derivative-new-metrics}
There are constants $(D_k)_{k\geq 0}$
such that $\|D^k\g^\ve\|_{C^0}<D_k\ve^k$ for all $\ve>0$ and $k\geq 0$.
\end{proposition}

Above, $\|D^k\g^\ve\|_{C^0}$ represents the maximum $k$-th derivative of the coefficients of the metric in the canonical basis, where the norm is taken with respect to the metric $\dd_\ve$.

\begin{proof}
Recall that $M$ has dimension $n$. Since the claims are local, we can fix a chart and assume we are on $\R^n$. In the sequel,
we fix $x_0$ in this neighborhood and estimate $\g^\ve$ around $x_0$. Applying a translation, we 
can further assume that $x_0=0$. Letting $B(1)\subset\R^n$ be the ball centered at 0 with radius 1,
consider the homothety $\psi:B(0,\ve f(0))\to B(1)$ defined by
$\psi(x)=x/\ve f(0)$, which is ``zooming in'' on a neighborhood of 0. The metric $\g^\ve$ is pushed under $\psi$ to a metric $\hatg^\ve$ on the unit ball of $\R^n$. By definition, for $y=\psi(x)\in B(1)$ and $v,w\in \R^n$ we have
$$
\hatg^\ve_y(v,w)=\g^\ve_{x}(\ve f(0)v,\ve f(0)w)=[\ve f(0)]^2\cdot \g^\ve_{x}(v,w)
=\left[\frac{f(0)}{f(x)}\right]^2g_x(v,w).
$$
Let $\{v_1,\ldots,v_n\}=\left\{\frac{\partial}{\partial x_1},\ldots,\frac{\partial}{\partial x_n}\right\}$ be the canonical basis,
$g_{ij}=g(v_i,v_j)$ and $\g^\ve_{ij}=\g^\ve(v_i,v_j)$ be the coefficients of $g$ an $\g^\ve$ respectively. The
coefficients of $\hatg^\ve$ are
\begin{equation}\label{equation-pushed-metric}
\hatg^\ve_{ij}(y)=\left[\frac{f(0)}{f(x)}\right]^2g_{ij}(x)\ \text{ where }y=\psi(x).
\end{equation}
Since $\tfrac{1}{8}\leq \tfrac{f(0)}{f(x)}\leq 8$ on $B(0,\ve f(0))$ and $g$ is a $C^\infty$ metric on a closed manifold,
there is $D_0>0$ such that $\|\g^\ve\|_{C^0}<D_0$ for all $\ve>0$.

Now we consider $k=1$, estimating the derivatives of the metric coefficients.
Recall that $x=\psi^{-1}(y)=\ve f(0)y$. For $\ell=1,\ldots,n$, let 
$\partial_\ell f=\frac{\partial f}{\partial x_\ell}$, $\partial_\ell g_{ij}=\frac{\partial g_{ij}}{\partial x_\ell}$ and 
$\partial_\ell \hatg^\ve_{ij}=\frac{\partial \hatg^\ve_{ij}}{\partial y_\ell}$.
Then
\begin{align}\label{equation-1st-derivative}
\partial_\ell \hatg^\ve_{ij}(y)&=f(0)^2\left[\frac{\partial_\ell g_{ij}(x)\ve f(0)\cdot f(x)^2-g_{ij}(x)\cdot 2f(x)\partial_\ell f(x)\ve f(0)}{f(x)^4}\right] \nonumber\\
&=\ve \left[\partial_\ell g_{ij}(x)\frac{f(0)^3}{f(x)^2}-
2g_{ij}(x)\partial_\ell f(x)\frac{f(0)^3}{f(x)^3}\right].
\end{align}
By Proposition \ref{lemma-f}(2), the first term in the brackets is bounded by
$8^2\|g\|_{C^1}\|f\|_{C^0}$ and the second is bounded by $2\cdot 8^3\|g\|_{C^1}C_1$,
hence $\exists D_1>1$ such that $\|D\g^\ve\|_{C^0}<D_1\ve$ for all $\ve>0$.

Before proceeding by induction, we make the calculation for the second order derivatives. Letting 
$\partial_{\ell_1\ell_2}\hatg^\ve_{ij}=\tfrac{\partial^2 \hatg^\ve_{ij}}{\partial y_{\ell_1}\partial y_{\ell_2}}$
(and $\partial_{\ell_1\ell_2}f, \partial_{\ell_1\ell_2}g_{ij}$ accordingly), we have:
\begin{align*}
\partial_{\ell_1\ell_2}\hatg^\ve_{ij}(y)&=\ve\left[
\partial_{\ell_1\ell_2}g_{ij}(x)\ve f(0)\frac{f(0)^3}{f(x)^2}+\partial_{\ell_1}g_{ij}(x)f(0)^3\left(-\frac{2f(x)\partial_{\ell_2} f(x) \ve f(0)}{f(x)^4}\right)\right.\\
&\ \ \ \ \ \ \ -2\partial_{\ell_2} g_{ij}(x)\ve f(0)\partial_{\ell_1}f(x)\frac{f(0)^3}{f(x)^3}
-2g_{ij}(x)\partial_{\ell_1\ell_2}f(x)\ve f(0) \frac{f(0)^3}{f(x)^3}\\
&\ \ \ \ \ \ \ \left. -2g_{ij}(x)\partial_{\ell_1}f(x)f(0)^3\left(-\frac{3f(x)^2\partial_{\ell_2} f(x)\ve f(0)}{f(x)^6}\right)\right]\\
&=\ve^2\left[
\partial_{\ell_1\ell_2}g_{ij}(x)\frac{f(0)^4}{f(x)^2}-2\partial_{\ell_1}g_{ij}(x)\partial_{\ell_2} f(x)\frac{f(0)^4}{f(x)^3} \right.\\
&\ \ \ \ \ \ \ \ -2\partial_{\ell_2} g_{ij}(x)\partial_{\ell_1}f(x)\frac{f(0)^4}{f(x)^3}
-2g_{ij}(x)\partial_{\ell_1\ell_2}f(x)\frac{f(0)^4}{f(x)^3}\\
&\ \ \ \ \ \ \ \ \left. +6g_{ij}(x)\partial_{\ell_1}f(x)\partial_{\ell_2} f(x)\frac{f(0)^4}{f(x)^4}\right].
\end{align*}
Calling the five terms inside the brackets by I, II, III, IV and V, Proposition \ref{lemma-f}(2) gives:
\begin{enumerate}[$\circ$]
\item $|{\rm I}|\leq 8^2\|g\|_{C^2}\|f\|_{C^0}^2$.
\item $|{\rm II+III}|\leq 4\cdot 8^3\|g\|_{C^2}\|f\|_{C^0}C_1$.
\item $|{\rm IV}|\leq \left|2g_{ij}(x)\cdot 2f(x)\partial_{\ell_1\ell_2}f(x)\frac{f(0)^3}{f(x)^3}\right|\leq 4\cdot 8^3\|g\|_{C^2}C_2$.
\item $|{\rm V}|\leq 6\cdot 8^4\|g\|_{C^2}C_1^2$.
\end{enumerate}
Hence $\exists D_2>1$ such that $\|D^2\g^\ve\|_{C^0}<D_2\ve^2$ for all $\ve>0$.

By induction, we assume that every $k$-th partial derivative of $\hatg^\ve_{ij}$ is the sum of a bounded number of terms of the form 
\begin{equation}\label{general-term}
C \ve^k\partial_{L_1}g_{ij}(x)\partial_{L_2}f(x)\cdots \partial_{L_m}f(x)\frac{f(0)^{k+2}}{f(x)^{m+1}}
\end{equation}
where $C$ is a constant (omitting the dependence on $k$) and $L_1,\ldots,L_m$ are $m\leq k$ multi-indices with $|L_1|+\cdots+|L_m|=k$.

This holds for $k=1,2$.
Differentiating (\ref{general-term}) with respect to one of the variables leads to a sum of $m+1$ terms of three different types:
\begin{enumerate}[$\circ$]
\item A term of the form
\begin{align*}
&\ C \ve^k[\partial_{\widetilde L_1}g_{ij}(x)\ve f(0)]\partial_{L_2}f(x)\cdots \partial_{L_m}f(x)\frac{f(0)^{k+2}}{f(x)^{m+1}}\\
&=C \ve^{k+1}\partial_{\widetilde L_1}g_{ij}(x)\partial_{L_2}f(x)\cdots \partial_{L_m}f(x)\frac{f(0)^{k+3}}{f(x)^{m+1}}
\end{align*}
where $\widetilde L_1,L_2,\ldots,L_m$ are multi-indices with $|\widetilde L_1|+|L_2|+\cdots+|L_m|=k+1$.
\item $m-1$ terms of the form 
\begin{align*}
&\ C \ve^k\partial_{L_1}g_{ij}(x)\partial_{L_2}f(x)\cdots[\partial_{\widetilde L_j}f(x)\ve f(0)] \cdots \partial_{L_m}f(x)\frac{f(0)^{k+2}}{f(x)^{m+1}}\\
&=C \ve^{k+1}\partial_{L_1}g_{ij}(x)\partial_{L_2}f(x)\cdots \partial_{\widetilde L_j}f(x)\cdots\partial_{L_m}f(x)\frac{f(0)^{k+3}}{f(x)^{m+1}}
\end{align*}
where $L_1,L_2,\ldots,\widetilde L_j,\cdots,L_m$ are multi-indices with $|L_1|+|L_2|+\cdots+|\widetilde L_j|+\cdots+|L_m|=k+1$.
\item A term of the form 
\begin{align*}
&\ C \ve^k\partial_{L_1}g_{ij}(x)\partial_{L_2}f(x)\cdots\partial_{L_m}f(x)f(0)^{k+2}\left[-\frac{(m+1)\partial_{L_{m+1}}f(x)\ve f(0)}{f(x)^{m+2}}\right]\\
&=\widetilde C \ve^{k+1}\partial_{L_1}g_{ij}(x)\partial_{L_2}f(x)\cdots\partial_{L_{m+1}}f(x)\frac{f(0)^{k+3}}{f(x)^{m+2}}
\end{align*}
where $L_1,L_2,\ldots,L_j,\cdots,L_{m+1}$ are multi-indices with $|L_1|+\cdots+|L_{m+1}|=k+1$.
\end{enumerate}
This concludes the induction. To conclude the proof, we note that by Proposition \ref{lemma-f}(2)
a term of the form (\ref{general-term}) is bounded by
\begin{align*}
&\ |C| \ve^k|\partial_{L_1}g_{ij}(x)|\cdot |f(x)^{|L_2|-1}\partial_{L_2}f(x)|\cdots |f(x)^{|L_m|-1}\partial_{L_m}f(x)|\frac{|f(0)|^{k+2}}{|f(x)|^{m+1+k-|L_1|-m+1}}\\
& \leq 2^{k+2}|C|\ve^{k} \|g\|_{C^{|L_1|}}C_{|L_2|}\cdots C_{|L_m|}\|f\|_{C^0}^{|L_1|}.
\end{align*}
The sum of finitely many terms of this form is bounded by $D_k\ve^k$ for some constant $D_k$.
This concludes the induction step and the proof of the proposition.
\end{proof}

Next, we estimate the injectivity radius of $\g^\ve$. This is related to the behaviour of geodesics, which we will study as  solutions of ordinary differential equations (ODE). The injectivity radius of $\g^\ve$ will then be estimated
by some uniform version of classical results on the lifetime of solutions of ODE's depending on a parameter.

Recall we fixed a chart and considered $x\in \R^n$ with the metric $\g^\ve$; by applying the homothety $\psi$, we obtain another metric $\hatg^\ve$ on $\R^n$. We write $B_1\subset\R^n$ for the ball with center at the origin and radius 1 in this metric.

\medskip
\noindent
{\sc Equation of geodesics for $\hatg^\ve:$} A geodesic for the metric $\hatg^\ve$ is a curve whose representation in charts is $\gamma(t)=(x_1(t),\ldots,x_n(t))$ satisfying the system of second order ODE's 
$$
\frac{d^2x_k}{dt^2}+\sum_{i,j}\Gamma_{ij}^k \frac{dx_i}{dt}\cdot \frac{dx_j}{dt}=0,
\ \ \ k=1,2,\ldots,n,
$$
where $\Gamma_{ij}^k$ are the {\em Christoffel symbols} of $\hatg^\ve$, defined by
\begin{equation}\label{Christoffel}
\Gamma_{ij}^k=\frac{1}{2}\sum_{m=1}^n \left\{\partial_i\hatg_{jm}^\ve+\partial_j\hatg_{mi}^\ve-\partial_m \hatg_{ij}^\ve\ \right\}(\hatg^\ve)^{mk}.
\end{equation}

\medskip
Above, $(\hatg^\ve)^{mk}$ denotes the coefficients of the inverse matrix of $(\hatg^\ve_{mk})$.
The Christoffel symbols have the following bounds.

\begin{lemma}\label{lemma-christoffel}
There is a constant $C>0$ such that for all $x\in \mm$, all $\ve>0$, and $i,j,k\in\{1,2,\ldots,n\}$, the Christoffel symbol $\Gamma_{ij}^k:B_1\to\R$ is $C^\infty$ with
$\|\Gamma_{ij}^k\|_{C^2}\leq C\ve$. 
\end{lemma}

\begin{proof}
The proof is a direct consequence of Proposition \ref{prop-derivative-new-metrics}, which 
actually gives that $\|D^\ell \Gamma_{ij}^k\|_{C^0}\leq {\rm const}\times \ve^{\ell+1}$.
Since we only need this for $\ell=0,1,2$, we leave the general case to the interested reader. 

By (\ref{equation-pushed-metric}), when $x_0$ ranges over $\mm$ (and $x$ ranges over $B(x_0,\ve f(x_0))$, $(\hatg^\ve_{ij})$ belongs to a compact region of ${\rm GL}(n,\R)$, so the same happens to their inverse matrices. By Proposition \ref{prop-derivative-new-metrics},
$\|\Gamma_{ij}^k\|_{C^0}\leq {\rm const}\times \ve$. 
Differentiating (\ref{Christoffel}) once and using Proposition \ref{prop-derivative-new-metrics} again,
we get that $\|D\Gamma_{ij}^k\|_{C^0}\leq {\rm const}\times \ve^2$. Repeating this procedure one more time,
we get that $\|D^2\Gamma_{ij}^k\|_{C^0}\leq {\rm const}\times \ve^3$. The proof is complete.
\end{proof}

In the sequel, we fix a small parameter $\ve_0>0$.
The previous lemma implies that, for each $i,j,k\in\{1,2,\ldots,n\}$, 
there is a $C^\infty$ map
$\ve\in(0,\ve_0]\to \Gamma_{ij}^k$, belonging to $C^1(B_1;\R)$, where $\Gamma_{ij}^k$ is the Christoffel symbol associated to the parameter $\ve$.

The equation of geodesics can be reduced to a first order ODE, by introducing the velocity
$y=(y_1,\ldots,y_n)=(\tfrac{dx_1}{dt},\ldots,\tfrac{dx_n}{dt})$
as a variable.
The defining function of this first order ODE, which depends on $\ve$, is
$$
F(x,y,\ve)=(y_1,\ldots,y_n,G_{\ve}^1,\ldots,G_{\ve}^n)
$$
where
$$
G_{\ve}^k(x,y)=-\sum_{i,j}\Gamma_{ij}^k y_iy_j.
$$
Extending $F$ to $\ve=0$ by setting $F(x,y,0)=(y_1,\ldots,y_n,0,\ldots,0)$, we have that $F$ is $C^2$ and, for each $x$, 
$$\|F(x,y_1,\ve)-F(x,y_2,0)\|_{C^0}\leq C\left[\|y_1-y_2\|+|\ve|\right].
$$
For simplicity, we write $F_\ve$ for $F(\cdot,\cdot,\ve)$.
Now we are in position to prove the positivity of the injectivity radius
of $\g^\ve$.

\begin{proposition}\label{prop-inj-radius}
If $\ve>0$ is small enough, then $\g^\ve$ has positive injectivity radius.
\end{proposition}

\begin{proof}
We consider solutions of the ODE with initial condition $(0,\ldots,0,y)=(0_n,y)$ at initial time $t=0$, where $y\in\R^n$. 
Let $\Phi(t,y,\ve)$ be the solution of this problem. The exponential map at $x_0$ near $0\in T_{x_0}M$ is then $\operatorname{exp}_{x_0}(ty)=\Phi(t,y,\ve)$ for all $0\le t\le t_0(y,\ve)$.

It is easy to see that there is some positive function $t_0(\ve)$ such that $\Phi(t,y,\ve)$ is well-defined for all $\|y\|\le 1$ and $0\leq t\leq t_0(\ve)$ (we omit the dependence on $x_0$). We assume that $t_0$ is maximal (possibly equal to $+\infty$). Note that $t_0(0)=+\infty$ as the solution is then constant. As long as the solution exists, we have the integral equation
$$
  \Phi(t,y,\ve)=(0_n,y)+\int_0^t F_\ve(\Phi(s,y,\ve))ds.
$$

We observe that the solutions associated to $F_0$ and to $F_\ve$ are both well-defined in $(-t_1,t_1)$, where $t_1:=t_0(\ve)$ if this is finite, and otherwise taking $t_1>0$ arbitrarily large. 
Therefore the difference function $\alpha(t)=\|\Phi(t,y,\ve)-\Phi(t,y,0)\|$ satisfies for all $0\le t\le t_1$
\begin{align*}
&\ \alpha(t)\leq \int_0^t \|F_{\ve}(\Phi(s,y,\ve))-F_{0}(\Phi(s,y,0))\|ds\\
&\leq C\int_0^{t}[\alpha(s)+|\ve|]ds
\leq Ct_1|\ve|+\int_0^t C\alpha(s)ds.
\end{align*}
By the integral Grönwall inequality, we obtain that 
$$
\forall 0\le t<t_1\quad \alpha(t)\leq Ct_1|\ve|e^{Ct}\leq \left(Ct_1e^{Ct_1}\right) |\ve|.
$$
In particular, if $t_1<\infty$, then $\alpha(t)$ hence $\Phi(t,t,\ve)$ stay bounded up to $t_1$, hence $t_0$ must be infinite.

Since the derivative of $\Phi$ also satisfies an ODE defined by a $C^1$ function (a derivative of $F_\varepsilon$),
an analogous argument
applies to the derivative map. We conclude that, there is some $\overline{C}$ depending only on the Lipschitz constant of $F_\ve$ such that
$$
\|\Phi(\cdot,\cdot,\ve)-\Phi(\cdot,\cdot,0)\|_{C^1}\leq 2\overline{C}|\ve|,
$$
where the norm is taken in $[0,t_1]\times B_1$. Since $\Phi(t_1,\cdot,0)$ is a diffeomorphism
from $B_1$ onto its image, the same is true of $\Phi(t_1,\cdot,\ve)$ when $\ve$ is small.

This shows that  if $\ve_0$ is small, then the injectivity radius of $\widehat{g}^\varepsilon$ is larger than $t_1$ for all $\ve<\ve_0$. Since $\g^\ve$ is isometric
to $\widehat{g}^\ve$ via the homothety $\psi$, we obtain the same property
for $\g^\ve$.
\end{proof}

Now we obtain estimates for $X$ with respect to $\g^\ve$. Recall
that $\mathfrak r(x)=\tfrac{1}{20L}\|X(x)\|$.

\begin{lemma}\label{lemma-estimate-X}
The following estimates hold for $X$.
\begin{enumerate}[{\rm (1)}]
\item $\tfrac{1}{2\ve}\leq \inf_{x\in\mm}\vertiii{X(x)}_{\ve}\leq \sup_{x\in\mm}\vertiii{X(x)}_{\ve}\leq \tfrac{2}{\ve}$.
\item $\vertiii{DX_x}_\ve=\|DX_x\|$ for all $x\in\mm$; in particular,
$\vertiii{DX_x}_\ve$ is uniformly bounded. 
\item There are $C,\mathfrak r>0$ such that for all $x\in\mm$ it holds
$\Hol{\beta}(DX) <C$ on ${\overline B}_\ve(x,\mathfrak r)$.
\end{enumerate}
\end{lemma}

Above, $\Hol{\beta}(DX)$ and the ball ${\overline B}_\ve(x,\mathfrak r)$ 
are considered with respect to the metric $\dd_\ve$.

\begin{proof}
(1) By Proposition \ref{lemma-f},
$$
\frac{1}{2\ve}\leq \vertiii{X(x)}_\ve=\frac{\|X(x)\|}{\ve f(x)} 
\leq \frac{2}{\ve}
$$
and the result follows.

\medskip
\noindent
(2) For $v\in T_xM^*$, we have
$$
\frac{\vertiii{DX_xv}_{\ve}}{\vertiii{v}_\ve}=\frac{\|DX_xv\|/\ve f(x)}{\|v\|/\ve f(x)}=\frac{\|DX_xv\|}{\|v\|}
$$
and so $\vertiii{DX_x}_\ve=\|DX_x\|$.

\medskip
\noindent
(3) Let $\mathfrak r=(160\ve L)^{-1}$. By Lemma \ref{lemma-complete}(2), we have that 
${\overline B}_\ve(x,\mathfrak r)\subset B(x,\mathfrak r(x))$.
By the Hölder continuity of $DX$ with respect to the original metric, 
there is $K>0$ such that
$$
\|DX_y-DX_z\|\leq Kd(y,z)^\beta,\text{ for all }y,z\in B(x,\mathfrak r(x)).
$$
Above, we are again assuming (by applying a chart) that $y,z\in \R^n$.
By (\ref{estimate-eps-norm}) and Lemma \ref{lemma-complete}(2), 
if $\dd_\ve(y,z)<\mathfrak r$ then
$$
\frac{\vertiii{DX_yv-DX_zv}_\ve}{\vertiii{v}_\ve}\leq 64 \frac{\|DX_yv-DX_zv\|}{\|v\|}\leq 64K d(y,z)^\beta \leq 256K(\ve f(x))^\beta\dd_\ve(y,z)^\beta.
$$
Observe that the Hölder constant of $DX$ actually improves near $\sing$. The proof follows with
$C=256K\ve^\beta \|f\|_{C^0}^\beta$.
\end{proof}

\begin{proof}[Proof of Theorem \ref{theorem-metric}]
Fix $\ve>0$ small enough, and let $\g=\g^\ve$. For $\mm$, we check properties (1)--(3) in the definition of tame manifold in page
\pageref{def-tame-manifold}:
\begin{enumerate}[$\circ$]
\item Property (1) follows from Lemma \ref{lemma-complete}(1).
\item Property (2) follows from Proposition \ref{prop-derivative-new-metrics}.
\item Property (3) follows from Proposition \ref{prop-inj-radius}.
\end{enumerate}
The relation between $d$ and $\dd$ stated in part (1) follows directly from Lemma \ref{lemma-complete}(2). Finally, part (2) 
follows from Lemma \ref{lemma-estimate-X}, which finishes
the proof of the theorem.
\end{proof}

\section{Symbolic dynamics for tame flows: proof of Theorem \ref{theorem-coding}}\label{section-theorem-coding}

\subsection{Recap from \cite{BCL23,LMN}}

In this section we recall the main ingredients used in the proofs of \cite{BCL23,LMN}, which assume $M$ to be a closed manifold, and then we explain the changes needed to obtain similar results for tame flows on tame manifolds.

We consider a tame manifold which we denote as previously as $(\mm,\g)$ (even if there is no longer any metric $g$) and a tame vector field $X$ on it, see the definitions in Section \ref{section-scheme}.
Let $\dd$ be the distance on $\mm$ induced by $\g$.
In the following, unless otherwise specified, all diameters of subsets of $\mm$, as well as all Hölder and Lipschitz norms of functions defined on $\mm$ or the exponential maps are taken with respect to the metric $\dd$. 
We denote balls in this metric by $\overline B(x,r)$.

By the tameness of $(\mm,\g)$, there are constants $\rho_1,\rho_2>0$ such that for every $x\in \mm$ the exponential map 
$\exp{x}:B(0,\rho_1)\to \mm$ is a diffeomorphism onto its image, which contains $\overline B(x,\rho_2)$, and with $C^2$ norm uniformly bounded by a constant that does not depend on $x$. Furthermore, these maps satisfy
assumptions (Exp1)-(Exp4) from \cite{BCL23,LMN}.

\medskip

\noindent
{\sc $\rho$-transverse disc:} A codimension one open disc $D \subset \mm$ is 
{\em $\rho$-transverse} if:
\begin{enumerate}[$\circ$]
    \item $D$ is compactly contained in a $C^\infty$ codimension one submanifold of $\mm$.
    \item $\mathrm{diam}(D) < 4 \rho$.
    \item For every $x \in D$, $\angle (X (x), T_x D^\perp ) < \rho$.
\end{enumerate}
We call $\rho$ the size of $D$. Every $\rho$-transverse disc $D$ defines a {\em flow box} $\varphi^{[-4 \rho, 4 \rho]} D$. 
If $\rho > 0$ is small enough, the map $(y,t) \in D \times [-4 \rho, 4 \rho] \mapsto \varphi^t (y)$ is a diffeomorphism onto the flow box $\varphi^{[-4 \rho, 4 \rho]} D$. We denote its inverse by $x \in \varphi^{[-4 \rho, 4 \rho]} D \mapsto (\mathfrak{q}_D (x), \mathfrak{t}_D (x))$, where $\mathfrak{q}_D: \varphi^{[-4 \rho, 4 \rho]} D \rightarrow D$ and $\mathfrak{t}_D: \varphi^{[-4 \rho, 4 \rho]} D \rightarrow [- 4 \rho, 4 \rho]$ are $C^{1+\beta}$.
We have the following uniform control on these maps.

\begin{lemma}\label{lemma-local-coord}
There are $\rho,L>0$ depending on $\mm$ and $X$ such that for every $\rho$-transverse discs $D, \ D'$, the above map $(y,t)\in D\times[-4\rho,4\rho]\mapsto \varphi^t(y)$ is a diffeomorphism satisfying:
\begin{enumerate}[{\rm (1)}]
    \item $\|\mathfrak q_{D}\|_{C^{1+\beta}} < L$ and  $\|\mathfrak t_{D}\|_{C^{1+\beta}}<L$.
    \item The map $\mathfrak{q}_D$ has a Lipschitz constant smaller than $2$.
    \item If $D'$ intersects the flow box $\varphi^{[-4 \rho, 4 \rho]} D$, then the restriction to $D'$ of the map $\mathfrak{t}_D$ has a Lipschitz constant smaller than $1$.
\end{enumerate}
\end{lemma}

\begin{proof}
For (1), the same proof as in \cite[Lemmas 2.2 and 2.3]{Lima-Sarig} applies, noting that constants defined there by minimum/maximum over finitely many charts are now bounded away from zero and infinity due to the uniformity of $C^2$ norms of the exponential maps and of the $C^{1+\beta}$ norm of $X$.  
Parts (2) and (3) are proved as in \cite[Lemma 2.1]{BCL23}.
\end{proof}

\medskip
\noindent
{\sc Global Poincaré section:} We say that the countable union $\Lambda = \bigcup_{i\geq 1} D_i$ of transverse discs $D_1,\ldots$ is a {\em global Poincaré section} if
there is $\rho>0$ such that
$$
\mm=\bigcup_{i\geq 1}\vf^{(0,\rho]}D_i.
$$
The {\em return time function} $r_{\Lambda}: \Lambda \rightarrow [0, \rho ]$ is defined by $r_\Lambda(x) := \inf \{ t>0 : \varphi^t (x) \in \Lambda \}$.

\begin{proposition}\label{prop-sections}
There are global Poincaré sections $\Lambda=\bigcup_{i\geq 1} D_i$ and $\Sh = \bigcup_{i\geq 1} E_i$
where $D_i,E_i$ are transverse discs of small size with $D_i\subset E_i$,
such that $d(\Lambda,\partial\Sh)>0$ and:
\begin{enumerate}[{\rm (1)}]
\item {\sc Partial order:} For all $i \neq j$, at least one of the sets $\overline{E_i} \cap \varphi^{[0, 4 \rho]} \overline{E_j}$ or $\overline{E_j} \cap \varphi^{[0, 4 \rho]} \overline{E_i}$ is empty; in particular $\overline{E_i} \cap \overline{E_j} = \emptyset$.
\item {\sc Uniform disjointness:} $\exists R_0>0$ such that $\inf r_\Sh>R_0$ and each $D_i$ is a $R_0$-transverse disc. 
\item {\sc Local finiteness:} $\exists N\geq 1$ such that $\#\{j\geq 1:\vf^{[-\rho,\rho]}E_i\cap E_j\not=\emptyset\}\leq N$, $\forall i\geq 1$.
\end{enumerate}
\end{proposition}

\begin{proof}
The proof for non-singular $C^{1+\beta}$ vector fields on a $C^\infty$ closed manifold $M$ was detailed in \cite[Lemma 2.7]{Lima-Sarig}; we follow the same approach, with some modifications as we now explain.

Let us summarize the proof of \cite[Lemma 2.7]{Lima-Sarig}, in our notation. It depends on four parameters $R_0<\rho_0\ll r_0\ll \rho_2$, as follows:

\smallskip
\noindent
Step 1: Cover $M$ by finitely many flow boxes of size $\rho_2$. In each of them, consider a $\rho_2$-transverse disc. The union of theses discs is a global Poincaré section, but the discs might intersect. 

\smallskip
\noindent
Step 2: Discretize each $\rho_2$-transverse disc by a net of points of thickness $\rho_0$.
Let $\mathfs L$ be the union of these nets. 

\smallskip
\noindent
Step 3: Shift each $x\in \mathfs L$ up or down in the flow direction by a parameter 
$\tau(x)\in [-r_0,r_0]$, so that the $R_0$-transverse discs centered at the shifted points satisfy properties (1)-(3). In this argument, the choice of $r_0$ depends on the total number $B$ of flow boxes.\footnote{In summary: $\rho_2=$ size of initial transverse discs, $\rho_0=$ thickness of net of points; $r_0=$ range of flow displacements; $R_0=$ size of displaced transverse discs.}

\medskip
\noindent
New Step 1:
In the tame setting, we cannot hope for a finite cover. However,
the tameness of $M^*$ and the uniform continuity of the vector field $X$ provide two important properties:
\begin{enumerate}[$\circ$]
\item The existence of a {\em countable locally finite} cover of $\mm$ by balls of radius $\rho_2$.
\item Let $\rho>0$ small and $D,D'$ be $\rho$-transverse discs; if $D\cap D'\neq\emptyset$
then $D,D'$ are ``close to parallel'' and so there is $\tau=o(\rho)$ such that $\vf^{\tau}(D)\cap D'=\emptyset$.
\end{enumerate}
These properties are consequences of the uniform control on the exponential maps, as explained in the beginning of this section.

\smallskip
\noindent
New Step 2: We can proceed as in the setting of a closed manifold.

\smallskip

\noindent
New Step 3:
As in the case for closed manifolds, we analyze the set of small parameters $t_i$  which define $R_0$-transverse discs $\varphi^{t_i}D_i$  disjoint from all other  discs $\varphi^{t_j}D_j$. Using the estimates given by Lemma \ref{lemma-local-coord}, by analyzing the relative positions of displaced $R_0$-transverse discs, one proves that the set of ``good'' parameters $t_i\in[-r_0,r_0]$ contains an interval of size $\approx R_0$, no matter the values of the other parameters $t_j$, $j\ne i$, making up $\mathfs L$. 

 When trying to implement this idea to $\mm$, the number of flows boxes is infinite, but considering a locally finite cover allows to define $r_0$ using $B$ as the local multiplicity; indeed, two $R_0$-transverse disc can only intersect if the flow boxes of size $\rho_2$ containing them intersect. 
We then fix 
$\{\overline B(x_i,\rho_2):i\geq 1\}$ a locally finite countable cover of $\mm$, where 
$B\geq 1$ satisfies
$$
\forall i\geq 1: \#\{j\geq 1:\overline B(x_j,\rho_2)\cap \overline B(x_i,\rho_2)\not=\emptyset\}\leq B.
$$
We can then reproduce steps 2 and 3 accordingly.
\end{proof}

\medskip
\noindent
{\sc One-form $\theta$:} There is a 1-form $\theta$ on $\mm$ such that:
\begin{enumerate}[{\rm (1)}]
    \item $\theta ( X(x)) = 1$ and $\angle ( X(x),\mathrm{Ker} (\theta_x)^{\perp}) < \rho, \forall x \in \mm$.
    \item $\mathrm{Ker}(\theta_x) = T_x \Sh, \forall x \in \Sh$.
\end{enumerate}

Write $N_x=\mathrm{Ker}(\theta_x)$, and 
consider the $(n-1)$-dimensional bundle $N := \bigsqcup_{x \in \mm}N_x$.
For each $x \in \mm$, let $\mathfrak{p}_x : T_x M^* \rightarrow N_x$ be the projection to $N_x$ parallel to $X(x)$.

\medskip
\noindent
{\sc Induced linear Poincaré flow:} The \emph{linear Poincaré flow of $\varphi$ induced by $\theta$} is the flow $\Phi = \{ \Phi^t \}_{t \in \R} : N \rightarrow N$ defined by $\Phi^t (v) = \mathfrak{p}_{\varphi^t (x)} [ D \varphi^t_x (v)]$ for $v \in N_x$.

\medskip
The construction of $\theta$ and its properties can be found in 
\cite[Section 2.4]{LMN}. Recall $\mm$ is endowed with the metric $\g$
given by Theorem \ref{theorem-metric}, whose norm is denoted
by $\vertiii{\cdot}$.

\medskip
\noindent
{\sc Non-uniformly hyperbolic locus $\nuh=\nuh(\vf,\chi,\rho,\theta)$:}  It is the $\vf$-invariant set of points $x\in \mm$ for which there is a splitting $N_x=N^s_x\oplus N^u_x$ such that:
\begin{enumerate}[(NUH1)]
\item For every $v\in N_x^s$:
$$\liminf_{t \to +\infty} \tfrac{1}{t} \log \vertiii{\Phi^{-t}v} > 0 \ \mbox{ and }\ \limsup_{t \to +\infty} \tfrac{1}{t} \log \vertiii{\Phi^{t}v} \leq-\chi.
$$

\item  For every $w\in N^u_x $:
    $$\liminf_{t \to +\infty} \tfrac{1}{t} \log \vertiii{\Phi^{t}w}> 0 \ \mbox{ and }\ \limsup_{t \to +\infty} \tfrac{1}{t} \log \vertiii{\Phi^{-t}w}\leq -\chi.
$$

\item The parameters $s(x) =\displaystyle\sup_{v \in N^s_x\atop{\vertiii{v}=1}} S(x,v)$  and $u(x) = \displaystyle\sup_{w\in N^u_{x}\atop{\vertiii{w}=1}} U(x,w)$ are finite, where:
\begin{align*}
&S(x,v)^2 = 4e^{2\rho}\int_{0}^ \infty e^{2\chi t } \vertiii{\Phi^{t}v}^2dt \ \text{ and}\ \
U(x,w)^2 =4e^{2\rho}\int_{0}^ \infty e^{2\chi t } \vertiii{\Phi^{-t}w}^2 dt.
\end{align*}
\end{enumerate}

\medskip
\noindent
{\sc Holonomy maps:} For each $x\in \Lambda$, the {\em (forward) holonomy map}
at $x$ is the continuous map $g_x^+=\vf^T$, where $T$ is continuous with $T(x)=r_\Lambda(x)$; the {\em (backward) holonomy map} is similarly
defined as $g_x^{-}=\vf^{-T}$ where $T$ is continuous with 
$T(x)=\inf\{t>0:\vf^{-t}(x)\in\Lambda\}$. 

\medskip
\noindent
{\sc Pesin chart $\Psi_x$ and parameter $Q(x)$:} Every $x\in\Lambda\cap\nuh$ has a
{\em Pesin chart} $\Psi_x:B(0,10Q(x))\subset \R^{n-1}\to \Sh$ where $Q(x)$ is a parameter
for which the charts representation $\Psi_{g_x^+(x)}^{-1}\circ g_x^+\circ\Psi_x$ of the holonomy map $g_x^+$
is a perturbation of a hyperbolic linear map on $B(0,10Q(x))$;
the same holds for $g_x^{-1}$. See \cite[Theorem 3.7]{LMN}. 

\medskip
\noindent
{\sc Hyperbolicity parameters $q(x), q^s(x), q^u(x) $:}
For $x \in \nuh$, define:
\begin{align*}
q(x) &:= \ve \inf \{ e^{\ve |t|} Q(\varphi^t(x)) : t \in \mathbb{R} \},\\
q^s(x)&:= \ve \inf \{ e^{\ve |t|} Q(\varphi^t(x)) : t \geq 0 \},\\
q^u(x)&:= \ve \inf \{ e^{\ve |t|} Q(\varphi^t(x)) : t \leq 0 \}.
\end{align*}

Recall that $\dd$ is the distance induced by $\overline{g}$.

\medskip

\noindent
{\sc Recurrently non-uniformly hyperbolic locus $ \nuh^\# = \nuh^\#(\varphi, \chi, \rho, \theta, \ve) $:} It is the invariant set of points \( x \in \text{NUH} \) such that:
\begin{enumerate}[(NUH3)]
    \item[(NUH4)] $q(x) > 0 $.
    \item[(NUH5)] $\limsup\limits_{t \to +\infty} \left[q(\varphi^t(x))\wedge \tfrac{1}{\dd(\vf^t(x),x)}\right] > 0  \mbox{ and }  \limsup\limits_{t \to -\infty} \left[q(\varphi^t(x))\wedge \tfrac{1}{\dd(\vf^t(x),x)}\right]> 0.$
\end{enumerate}

\medskip
Alternatively, $x\in\nuh^\#$ if and only if $Q$ does not decrease to zero with exponential rate faster than $\ve$ and there are $t_k,s_k\to+\infty$ and $\delta>0$ 
such that $q(\vf^{t_k}(x)),q(\vf^{-s_k}(x))>\delta$ and $\vf^{t_k}(x),\vf^{-s_k}(x)$
belong to a compact subset of $\mm$, for all $k$.

\subsection{A detailed statement}

In this section, we state and prove a detailed version of 
Theorem \ref{theorem-coding}. We begin defining a Bowen relation for flows, as introduced
in \cite{BCL23} (the original notion, for diffeomorphisms, was introduced in \cite{Boyle-Buzzi}).

\newcommand\hS{\widehat S}

Let  $T_r: S_r\to S_r$ be a suspension flow over a symbolic system $S$ that is an extension
of some flow $U:X\to X$ by a semiconjugacy map $\pi:S_r\to X$,  i.e.
$U^t\circ\pi=\pi\circ T^t_r$ for all $t\in\R$.\\

\noindent
{\sc Bowen relation:} A \emph{Bowen relation} $\sim$ for $(T_r,\pi,U)$ is a symmetric binary
relation on the alphabet of $S$ satisfying the following two properties:
\begin{enumerate}[i,]
\item[{\rm (i)}] $\forall\omega,\omega'\in S_r,\;\; \pi(\omega)=\pi(\omega')\implies \operatorname{v}(\omega)\sim\operatorname{v}(\omega')$, where $\operatorname{v}(x,t):=x_0$ for $x\in S$;
\item[{\rm (ii)}] $\exists \gamma>0$ with the following property:
$$\forall\omega,\omega'\in S_r,\;\; \left[ \forall t\in\R,\; \operatorname{v}(T_r^t\omega)\sim\operatorname{v}(T_r^t\omega') \right] \implies \left[ \exists |t|<\gamma\text{ s.t. } \pi(\omega)= U^t(\pi(\omega')) \right].
$$
  \end{enumerate}

\begin{theorem}\label{t.main.tame}
Let $X$ be a tame $C^{1+\beta}$ vector field ($\beta>0$) on a tame $C^\infty$ Riemannian manifold $\mm$.
For each $\chi>0$, there exists a locally compact topological Markov flow
$(\widehat \Sigma_{\widehat r},\widehat\sigma_{\widehat r})$ with graph
$\widehat{\mathfs G}=(\widehat V,\widehat E)$ and roof function $\widehat{r}$,
and a map $\widehat \pi_{\widehat r}:\widehat \Sigma_{\widehat r}\to \mm$ such that
$\widehat \pi_{\widehat r}\circ {\widehat\sigma}_{\widehat r}^t=\vf^t\circ\widehat \pi_{\widehat r}$ for all $t\in\R$, and satisfying:
\smallskip
\begin{enumerate}[{\rm (1)}]
\item $\widehat r$ and $\widehat \pi_{\widehat r}$ are H\"older continuous.
\smallskip

\item $\widehat \pi_{\widehat r}[\widehat \Sigma_{\widehat r}^\#]=\nuh^\#$ has full measure for every $\chi$-hyperbolic measure; for every ergodic $\chi$-hyperbolic measure $\mu$,
there is an ergodic $\widehat\sigma_{\widehat r}$-invariant measure $\overline \mu$ 
on $\widehat \Sigma_{\widehat r}$
such that $(\widehat \pi_{\widehat r})_*\overline \mu=\mu$ and $h_{\overline{\mu}}(\widehat \sigma_{\widehat{r}})=h_\mu(\vf)$.
\smallskip

\item If
$(\un R,t)\in \widehat \Sigma_{\widehat r}^{\#}$ satisfies
$R_n=R$ and $R_m=S$ for infinitely many $n<0$ and $m>0$, then $\operatorname{Card}\{z\in \widehat \Sigma_{\widehat r}^\#:\widehat \pi_{\widehat r}(z)=\widehat \pi_{\widehat r}(\un R,t)\}$
is bounded by a number $C(R,S)$, depending only on $R,S$.
\smallskip

\item\label{i.splitting} There is $\lambda>0$ and for $x\in \widehat \pi_{\widehat r}(\widehat \Sigma_{\widehat r})$ there is a unique splitting
$N_{x}=N^s_{x}\oplus N^u_{x}$ such that:
\begin{align*}
\limsup_{t\to +\infty} \tfrac{1}{t}\log \vertiii{\Phi^t|_{N^s_{x}}}\leq -\lambda
\quad \text{ and }\quad\limsup_{t\to +\infty} \tfrac{1}{t}\log\vertiii{\Phi^{t}|_{N^s_{\varphi^{-t}(x)}}}\leq -\lambda\\
\quad \limsup_{t\to +\infty} \tfrac{1}{t}\log\vertiii{\Phi^{-t}|_{N^u_{x}}}\leq -\lambda
\quad \text{ and }\quad\limsup_{t\to +\infty} \tfrac{1}{t}\log \vertiii{\Phi^{-t}|_{N^u_{\varphi^t(x)}}}\leq -\lambda.
\end{align*}
The splitting is $\Phi$-equivariant, and the maps $z\mapsto N^{s/u}_{\widehat \pi_{\widehat r}(z)}$ are H\"older continuous on $\widehat \Sigma_{\widehat r}$.
\smallskip

\item\label{i.manifold} There is $\alpha>0$ and for every $z\in \widehat \Sigma_{\widehat r}$
there are $C^1$ submanifolds $V^{cs}(z),V^{cu}(z)$ passing through $x:=\widehat \pi_{\widehat r}(z)$
such that:
\begin{enumerate}[{\rm (a)}]
\item
$T_{x}V^{cs}(z)=N^s_x+\mathbb{R}\cdot X(x)$ and $T_{x}V^{cu}(z)=N^{u}_x+\mathbb{R}\cdot X(x)$.
\item For all $y\in V^{cs}(z)$, there is $\tau\in \mathbb{R}$ such that
$\dd(\varphi^t(x),\varphi^{t+\tau}(y))\leq e^{-\alpha t}$, $\forall t\geq 0$.
\item For all $y\in V^{cu}(z)$, there is $\tau\in \mathbb{R}$ such that
$\dd(\varphi^{-t}(x),\varphi^{-t+\tau}(y))\leq e^{-\alpha t}$, $\forall t\geq 0$.
\end{enumerate}
\smallskip

\item\label{i.Bowen} There is a symmetric binary relation $\sim$ on the alphabet $\widehat V$
satisfying:
\begin{enumerate}[{\rm (a)}]
\item For any $R\in \widehat V$, the set $\{S\in \widehat V:R\sim S\}$ is finite.
\item The relation $\sim$ is a Bowen relation for $(\widehat\sigma_{\widehat r},\widehat\pi_{\widehat r}|_{\widehat\Sigma^\#_{\widehat r}},\vf^t)$.
\end{enumerate}
\smallskip

\item\label{i.canonical} There exists a measurable set $\mathfs R$ with a measurable partition
indexed by $\widehat V$, which we denote by $\{R:R\in\widehat V\}$, such that:
\begin{enumerate}[{\rm (a)}]
\item The orbit of any point $x\in \nuh^\#$ intersects $\mathfs R$.
\item The first return map $H\colon \mathfs R\to \mathfs R$ induced by $\varphi$ is a well-defined bijection.
\item For any $x\in \mathfs R$, if $\un R=\{R_n\}_{n\in\Z}$ satisfies $H^n(x)\in R_n$ for all $n\in\Z$,
then $(\un R,0)\in \widehat \Sigma^\#_{\widehat r}$ and $\widehat \pi_{\widehat r}(\un R,0)=x$.
\end{enumerate}

\item\label{i.lift}
For any compact transitive invariant hyperbolic set $K\subset \mm$ whose ergodic $\vf$-invariant
measures are all $\chi$-hyperbolic, there is a transitive invariant compact set 
$X\subset \widehat \Sigma_{\widehat r}$ such that $\widehat\pi_{\widehat r}(X)=K$.
\end{enumerate}
\end{theorem}

For flows in closed manifolds, this is Theorem 9.1 in \cite{BCL23,LMN}.
The relation $\sim$ is the \emph{affiliation}. 
Note that the assumption 
$\bigl[\operatorname{v}(\widehat \sigma_{\widehat r}^t(z))\sim \operatorname{v}(\widehat\sigma_{\widehat r}^t(z'))$
for all $t\in \mathbb{R}\bigr]$ consists of countably many
affiliation conditions: if $z=(\un R,s)$ and $z'=(\un S,s')$, then varying $t$ in the interval $[\widehat r_n(\un R),\widehat r_{n+1}(\un R))$
provides $i\leq \tfrac{\sup(\widehat r)}{\inf(\widehat r)}$ affiliations of the form
$R_n\sim S_{m+1},\ldots,R_n\sim S_{m+i}$.

Part~\eqref{i.canonical} provides for any
$x\in  \nuh^\#$ a particular pair $(\un R,t)\in \widehat \Sigma_{\widehat r}^\#$
such that $\widehat \pi_{\widehat r}(\un R,t)=x$
($t$ is the smallest non-negative number such that $\varphi^{-t}(x)\in \mathfs R$). We call
the pair $(\un R,t)$ the \emph{canonical lift} of $x$.
This is a measurable embedding of $\nuh^\#$ into $\widehat\Sigma_{\widehat r}$.

Part~\eqref{i.lift} is a version of \cite[Proposition 3.9]{BCS-MME} in our context, and the proof is the same of \cite[Section 9.4]{LMN}.

\begin{proof}[Proof of Theorem \ref{t.main.tame}]
We follow ipsis litteris the arguments of 
\cite{BCL23,LMN}. The main difference is that, while these later
works consider only closed manifolds, we allow for tame non-compact
manifold, so all global arguments that require the compactness
of $M$ have to be adapted. We do that by applying analogous 
arguments to an exhaustion of $\mm$ by compact balls with increasing radii, paying attention that the tameness of $\mm$ and $X$ give uniform
estimates. In the sequel, we summarize the main steps in the proof and point
when changes are needed.

Recall we have fixed:
\begin{enumerate}[$\circ$]
\item Two global Poincaré sections $\Lambda\subset\Sh$, see Proposition \ref{prop-sections}.
\item A one-form $\theta$ and its associated induced linear Poincaré flow
$\Phi$. 
\item Two parameters $\chi,\ve>0$.
\end{enumerate}

We can assume $\mm$ is connected (in the general case, we apply the construction to every connected component of $\mm$). In the sequel, we fix a point $x^*\in\mm$.

\medskip
\noindent
{\sc Step 1 (General definitions):} Introduce the overlap condition $\Psi_x^\eta\overset{\ve}{\approx}\Psi_y^{\eta'}$, $\ve$-double charts $\Psi_x^{p^s,p^u}$, the edge condition $\Psi_x^{p^s,p^u}\to \Psi_y^{q^s,q^u}$, $\ve$-generalized pseudo-orbits ($\ve$-gpo), and graph transforms associated to edges.

\medskip
This step is done exactly as in \cite{BCL23,LMN}.

\medskip
\noindent
{\sc Step 2 (Coarse graining):} Construct a countable family $\mathfs A$ of $\ve$-double charts with the following properties:
\begin{enumerate}[{\rm (1)}]
\item {\sc Discreteness:} For all $t>0$, the set $\{\Psi_x^{p^s,p^u}\in\mathfs A:p^s,p^u,\dd(x,x^*)^{-1}>t\}$ is finite.
\item {\sc Sufficiency:} If $x\in\Lambda\cap\nuh^\#$ then there is a regular $\ve$-gpo
$\un v\in{\mathfs A}^{\Z}$ that shadows $x$.
\item {\sc Relevance:} For each $v\in \mathfs A$, $\exists\un{v}\in\mathfs A^\Z$ an $\ve$-gpo
with $v_0=v$ that shadows a point in $\Lambda\cap\nuh^\#$.
\end{enumerate}

\medskip
The main difference from \cite{BCL23,LMN} is that the discreteness assumption
also controls $\dd(x,x^*)$  from above. Following \cite[Proof of Theorem 5.1]{LMN},
we obtain this by adding an extra parameter $h\in\N$ and instead of 
considering $Y_{\un \ell,m,j,k}$ we let 
$$
Y_{\un \ell,m,j,k,h}:=\left\{\Gamma(x)\in Y:
\begin{array}{cl}
e^{\ell_i}\leq\vertiii{C(f^i(x))^{-1}}<e^{\ell_i+1},&-1\leq i\leq 1\\
e^{-m-1}\leq Q(x)< e^{-m}&\\
e^{-j-1}\leq q(x)< e^{-j}&\\
\dim N_x^s = k\text{ and }\dd(x,x^*)<e^h
\end{array}
\right\}.
$$
This allows to recover the pre-compactness of these
sets, and the rest of the proof follows. 

Let $\Sigma$ be the TMS with vertex set $\mathcal A$ and edge set given by the edge condition. Then $\Sigma$ is locally compact (every vertex has finite ingoing and outgoing degree).

\medskip
\noindent
{\sc Step 3 (First coding $\pi$):} Define the coding map $\pi:\Sigma\to\Sh$, where 
$\pi(\un v)=$ unique point shadowed by $\un v$; this map is Hölder
continuous and satisfies $\pi(\Sigma^\#)\supset\Lambda\cap\nuh^\#$.

\medskip
The definition and properties of $\pi$ hold as stated in \cite{BCL23,LMN}.

\medskip
\noindent
{\sc Step 4 (Locally finite countable cover $\mathfs Z$):} The countable family $\mathfs Z=\{Z(v):v\in\mathfs A\}$
defined by $Z(v)=\{\pi(\un v):\un v\in\Sigma^\# \text{ with }v_0=v\}$ 
is locally finite with
$\Lambda\cap\nuh^\#\subset \mathfs Z\subset\Sh\cap\nuh^\#$, 
and satisfies a symbolic Markov property.

\medskip
Above, we identify $\mathfs Z$ with the union of its elements. Step 4 relies on an {\em inverse theorem}: if 
$x=\pi(\un v)$ with $\un v=\{\Psi_{x_n}^{p^s_n,p^u_n}\}_{n\in\Z}\in\Sigma^\#$, then the hyperbolicity parameters of $x$ and $x_0$ are close. The precise statement \cite[Theorem 6.1]{LMN} is technical, so we have decided not to reproduce it here. Let us just remark
why $\pi[\Sigma^\#]\subset \nuh^\#$, since the defining condition 
(NUH5) is different from \cite{LMN}. The same proof of \cite[Prop. 6.6]{LMN}
gives that $x\in\nuh$. Now, by the inverse theorem \cite[Theorem 6.1]{LMN},
there is a sequence
$t_n\to+\infty$ such that $\dd(\vf^{t_n}(x),x_n)\ll 1$ and $p^{s/u}_n\approx p^{s/u}(\vf^{t_n}(x))$ for every $n\geq 0$. Using that $\un v\in\Sigma^\#$, it follows that $x\in\nuh^\#$.   

\medskip
\noindent
{\sc Step 5 (Partition $\mathfs R$ and second coding $\wh\pi$):}
By the local finiteness, $\mathfs Z$ refines to a 
{\em countable} family $\mathfs R$ with the Markov property. The family $\mathfs R$
defines a TMF $\wh\Sigma_{\wh r}$ and a coding map $\wh\pi_{\wh r}:\wh\Sigma_{\wh r}\to\mm$.\\

We note that $\nuh^\#$ indeed carries all $\chi$-hyperbolic measures: the proof is the same of \cite[Prop. 3.1 and 3.4]{LMN}, observing that recurrence implies (NUH5) for $\mu$-a.e. $x\in\mm$.

\medskip

Proceeding as in \cite[Section 9]{LMN}, we conclude the proof of Theorem
\ref{t.main.tame}.
\end{proof}

\section{Proof of the Main Theorem}
In this section, we state and prove a detailed version of the Main Theorem (Theorem~\ref{t.main.singular}) and we explain how to derive it from Theorems~\ref{theorem-metric} and~\ref{t.main.tame}.

\subsection{Another detailed statement}
Given a vector field $X$, we define at any point $x\in M\setminus \sing$ the normal tangent space $\mathcal{N}_x=X(x)^\perp\subset T_xM$.
For $t\in \mathbb{R}$, the \emph{linear Poincar\'e flow} $\Psi=\{\Psi^t\}_{t\in\R}$ associates to any vector $v\in\mathcal{N}_x$ the orthogonal projection of $D\varphi_x^t(v)$ on $\mathcal{N}_{\varphi^t(x)}$:
$$\Psi^t(v)=D\varphi^t_x(v)-\frac{g(X(\varphi^t(x)),D\varphi^t_x(v))}{\|X(\varphi^t(x))\|^2}X(\varphi^t(x)).$$
In the terminology of Section~\ref{section-theorem-coding}, it is the linear Poincar\'e flow induced by the linear form $g(X,\cdot)$.
The \emph{rescaled linear Poincar\'e flow} $\overline\Psi=\{\overline\Psi^t\}_{t\in\R}$ is defined by conjugating by the function $\|X\|$:
$$\overline\Psi^t(v)=\|X(\varphi^t(x))\|^{-1}\cdot \Psi^t(\|X(x)\|\cdot v).$$
The following result implies the main theorem.

\begin{theorem}\label{t.main.singular}
Let $X$ be a $C^{1+\beta}$ vector field ($\beta>0$) on a $C^\infty$ closed Riemannian manifold $M$.
For each $\chi>0$, there exists a locally compact topological Markov flow
$(\widehat \Sigma_{\widehat r},\widehat\sigma_{\widehat r})$ with graph
$\widehat{\mathfs G}=(\widehat V,\widehat E)$ and roof function $\widehat{r}$,
and a map $\widehat \pi_{\widehat r}:\widehat \Sigma_{\widehat r}\to M\setminus {\rm Sing}$ which satisfy
$\widehat \pi_{\widehat r}\circ {\widehat\sigma}_{\widehat r}^t=\vf^t\circ\widehat \pi_{\widehat r}$ for all $t\in\R$, such that:
\smallskip
\begin{enumerate}[{\rm (1)}]
\item\label{i.holder.singular} $\widehat r$ is H\"older continuous, 
$\wh\pi_{\wh r}$ is $X$-scaled Hölder continuous: there are $C,\kappa>0$ such that
$$
d(\wh\pi_{\wh r}(z),\wh\pi_{\wh r}(w))\leq C\|X({\wh\pi}_{\wh r}(z))\| d_{\wh r}(z,w)^\kappa,\ \text{for all }z,w\in\widehat\Sigma_{\wh r}.
$$

\item\label{i.measures.singular} $\widehat \pi_{\widehat r}[\widehat \Sigma_{\widehat r}^\#]$ has full measure for every $\chi$-hyperbolic measure with $\mu(\sing)=0$; for every ergodic $\chi$-hyperbolic measure $\mu$ with $\mu(\sing)=0$,
there is an ergodic $\widehat\sigma_{\widehat r}$-invariant measure $\overline \mu$ 
on $\widehat \Sigma_{\widehat r}$
such that $(\widehat\pi_{\widehat r})_*\overline \mu =\mu$ and $h_{\overline{\mu}}(\widehat \sigma_{\widehat{r}})=h_\mu(\vf)$.
\smallskip

\item If
$(\un R,t)\in \widehat \Sigma_{\widehat r}^{\#}$ satisfies
$R_n=R$ and $R_m=S$ for infinitely many $n<0$ and $m>0$, then $\operatorname{Card}\{z\in \widehat \Sigma_{\widehat r}^\#:\widehat \pi_{\widehat r}(z)=\widehat \pi_{\widehat r}(\un R,t)\}$
is bounded by a number $C(R,S)$, depending only on $R,S$.
\smallskip

\item\label{i.splitting.singular} There is $\lambda>0$ and for {$x\in \widehat \pi_{\widehat r}(\widehat \Sigma_{\widehat r}^\#)$} there is a splitting
$\mathcal{N}_{x}=\mathcal{N}^s_{x}\oplus \mathcal{N}^u_{x}$ such that:
\begin{align*}
\limsup_{t\to +\infty} \tfrac{1}{t}\log \|{\overline\Psi^t|_{\mathcal{N}^s_{x}}}\|\leq -\lambda
\quad \text{ and }\quad{\limsup_{t\to +\infty}} \tfrac{1}{t}\log\|{\overline\Psi^{t}|_{\mathcal{N}^s_{\varphi^{-t}(x)}}}\|\leq -\lambda\\
\quad \limsup_{t\to +\infty} \tfrac{1}{t}\log\|{\overline\Psi^{-t}|_{\mathcal{N}^u_{x}}}\|\leq -\lambda
\quad \text{ and }\quad{\limsup_{t\to +\infty}} \tfrac{1}{t}\log \|{\overline\Psi^{-t}|_{\mathcal{N}^u_{\varphi^{t}(x)}}}\|\leq -\lambda.
\end{align*}
The splitting is unique, $\Psi$-equivariant, and the maps $x\mapsto \mathcal{N}^{s/u}_{\widehat \pi_{\widehat r}(x)}$ are $X$-scaled H\"older continuous on $\widehat \Sigma_{\widehat r}$.
\smallskip

\item\label{i.manifold.singular} There are $C,\alpha>0$ and for every $z\in \widehat \Sigma_{\widehat r}$
there are $C^1$ submanifolds $V^{cs}(z),V^{cu}(z)$ passing through $x:=\widehat \pi_{\widehat r}(z)$
such that:
\begin{enumerate}[{\rm (a)}]
\item
$T_{x}V^{cs}(z)=\mathcal{N}^s_x+\mathbb{R}\cdot X(x)$ and $T_{x}V^{cu}(z)=\mathcal{N}^{u}_x+\mathbb{R}\cdot X(x)$.
\item For all $y\in V^{cs}(z)$, there is $\tau\in \mathbb{R}$ such that
$$d(\varphi^t(x),\varphi^{t+\tau}(y))\leq C\|X(\vf^t(x))\|e^{-\alpha t}, \ \forall t\geq 0.
$$
\item For all $y\in V^{cu}(z)$, there is $\tau\in \mathbb{R}$ such that
$$
d(\varphi^{-t}(x),\varphi^{-t+\tau}(y))\leq C\|X(\vf^t(x))\|e^{-\alpha t},\ \forall t\geq 0.
$$
\end{enumerate}
\smallskip

\item\label{i.Bowen.singular} There is a symmetric binary relation $\sim$ on the alphabet $\widehat V$
satisfying:
\begin{enumerate}[{\rm (a)}]
\item For any $R\in \widehat V$, the set $\{S\in \widehat V:R\sim S\}$ is finite.
\item The relation $\sim$ is a Bowen relation for $(\widehat\sigma_{\widehat r},\widehat\pi_{\widehat r}|_{\widehat\Sigma^\#_{\widehat r}},\vf^t)$.
\end{enumerate}
\smallskip

\item\label{i.canonical.singular} There exists a measurable set $\mathfs R$ with a measurable partition
indexed by $\widehat V$, which we denote by $\{R:R\in\widehat V\}$, such that:
\begin{enumerate}[{\rm (a)}]
\item The orbit of any point $x\in \widehat \pi_{\widehat r}[\widehat \Sigma_{\widehat r}^\#]$ intersects $\mathfs R$.
\item The first return map $H\colon \mathfs R\to \mathfs R$ induced by $\varphi$ is a well-defined bijection.
\item For any $x\in \mathfs R$, if $\un R=\{R_n\}_{n\in\Z}$ satisfies $H^n(x)\in R_n$ for all $n\in\Z$,
then $(\un R,0)\in \widehat \Sigma^\#_{\widehat r}$ and $\widehat \pi_{\widehat r}(\un R,0)=x$.
\end{enumerate}

\item\label{i.lift.singular}
For any compact transitive invariant hyperbolic set $K\subset M$ whose ergodic $\vf$-invariant
measures are all $\chi$-hyperbolic {with $\mu(\sing)=0$}, there is a transitive invariant compact set 
$X\subset \widehat \Sigma_{\widehat r}$ such that $\widehat\pi_{\widehat r}(X)=K$.
\end{enumerate}
\end{theorem}

\begin{proof}[Proof of Theorem \ref{t.main.singular}]
We consider $M^*=M\setminus \sing$ endowed with a metric $\overline g$ as given by Theorem~\ref{theorem-metric}.
We let $\theta,\Phi,\nuh,\nuh^\#$ as built and defined in the previous section.
Apply Theorem \ref{t.main.tame} to $X$ considered as a vector field on $(\mm,\overline{g})$. 
Parts (3), (6) and (7) are direct.
We focus on the remaining ones.

\medskip

\paragraph{\bf Part~\eqref{i.holder.singular}.} The Hölder regularity of $\wh r$ is 
direct. For the other, apply 
Theorem \ref{theorem-metric}(1) and Theorem \ref{t.main.tame}(1) to get that
$$
d({\wh\pi}_{\wh r}(z),{\wh\pi}_{\wh r}(w))
\leq {\rm const}\times \|X({\wh\pi}_{\wh r}(z))\|\dd({\wh\pi}_{\wh r}(z),{\wh\pi}_{\wh r}(w))\leq
{\rm const}\times \|X({\wh\pi}_{\wh r}(z))\|d_{\wh r}(z,w)^\kappa$$
for some $\kappa\in(0,1)$.
\medskip

\paragraph{\bf Part~\eqref{i.measures.singular}.}
By Remark \ref{rmk-exponents}, the set of $\chi$-hyperbolic measures for $(\mm,\g)$ coincides with the set of $\chi$-hyperbolic measures for $(M,g)$ with $\mu(\sing)=0$. 
Then apply Theorem \ref{t.main.tame}(2).
\medskip

{
\paragraph{\bf Part~\eqref{i.splitting.singular}.}
The flows $\Phi$ and $\overline\Psi$ are conjugate, hence Theorem \ref{t.main.tame}(4) implies that for $x\in \widehat \pi_{\widehat r}(\widehat \Sigma_{\widehat r})$:
\begin{align*}
\limsup_{t\to +\infty} \tfrac{1}{t}\log \vertiii{{\overline\Psi^t|_{\mathcal{N}^s_{x}}}}\leq -\lambda
\quad \text{ and }\quad\limsup_{t\to +\infty} \tfrac{1}{t}\log\vertiii{{\overline\Psi^{t}|_{\mathcal{N}^s_{\varphi^{-t}(x)}}}}\leq -\lambda\\
\quad \limsup_{t\to +\infty} \tfrac{1}{t}\log\vertiii{{\overline\Psi^{-t}|_{\mathcal{N}^u_{x}}}}\leq -\lambda
\quad \text{ and }\quad\limsup_{t\to +\infty} \tfrac{1}{t}\log \vertiii{{\overline\Psi^{-t}|_{\mathcal{N}^u_{\varphi^{t}(x)}}}}\leq -\lambda.
\end{align*}
Assuming furthermore that $x\in \widehat \pi_{\widehat r}(\widehat \Sigma_{\widehat r}^\#)$, by recurrence the above estimates imply the first part of part (4).
Finally, the proof that $z\mapsto N^{s/u}_{\widehat \pi_{\widehat r}(z)}$ are $X$-scaled H\"older continuous on $\widehat \Sigma_{\widehat r}$ is the same proof of
Part \eqref{i.holder.singular}.
}

\medskip

\paragraph{\bf Part~\eqref{i.manifold.singular}.}
By Theorem \ref{t.main.tame}(5) and Theorem \ref{theorem-metric}(1),
if $y\in V^{cs}(z)$ then there is $\tau\in \mathbb{R}$ such that
$$
d(\varphi^t(x),\varphi^{t+\tau}(y))\leq 
C \|X(\vf^t(x))\|\dd(\varphi^t(x),\varphi^{t+\tau}(y))
\leq C \|X(\vf^t(x))\|e^{-\alpha t}
$$
for all $t\geq 0$. The other inequality is proved similarly.

\medskip

\paragraph{\bf Part~\eqref{i.lift.singular}.}
We argue as for Part~\eqref{i.measures.singular}, applying Theorem \ref{t.main.tame}(8).
\end{proof}

\section{Finiteness on the number of ergodic measures of maximal entropy}\label{section-finiteness}

In this section, we prove Theorem \ref{Thm:finiteness-C-infty}.
Let $M$ be a $C^\infty$ closed Riemannian manifold,
let $X$ be a $C^{1+\beta}$ vector field ($\beta>0$) on $M$, and $\vf=\{\vf^t\}_{t\in\R}$ the flow generated by $X$.

\subsection{Homoclinic relation}\label{sec.homoclinic}

The notion of homoclinic relation for ergodic hyperbolic measures was defined in \cite{BCS-MME} for diffeomorphisms and in \cite{BCL23} for flows. Let us recall the later one.
Let $\mu$ be a $\vf$-invariant probability measure with $\mu(\sing)=0$.
If $\mu$ is hyperbolic, then by Pesin theory $\mu$-a.e. $x$ has stable and unstable manifolds $W^s(x)$ and $W^u(x)$, which are one-dimensional immersed submanifolds of $M$. We let 
$$W^{cs}(x)=\bigcup_{t\in \mathbb{R}} W^s(\varphi^t(x))\ \text{ and }\ W^{cu}(x)=\bigcup_{t\in \mathbb{R}} W^u(\varphi^t(x))$$
denote the weak stable/unstable manifolds at $x$, respectively.
Let $\mu_1,\mu_2$ be $\vf$-invariant hyperbolic probability measures with 
$\mu_1(\sing)=\mu_2(\sing)=0$.

\medskip
\noindent
{\sc Homoclinic relation of measures:} We say that $\mu_1$ is \textit{homoclinically related to} $\mu_2$ if there are measurable sets $\Lambda_1, \Lambda_2$ with $\mu_1(\Lambda_1)>0$ and $\mu_2(\Lambda_2)>0$ such that
$$
x_1\in\Lambda_1,x_2\in\Lambda_2\ \Longrightarrow\ W^{cs}(x_1) \pitchfork W^{cu}(x_2) \neq \emptyset\ \text{ and }\ W^{cu}(x_1) \pitchfork W^{cs}(x_2) \neq \emptyset.
$$ 

\medskip
Above, we use the definition that two submanifolds $V_1,V_2\subset M$ intersect
transversely, and write $V_1 \pitchfork V_2$, if there is $x\in V_1\cap V_2$ such that $T_xM=T_xV_1+T_xV_2$.

The following result shows that two distinct ergodic measures of maximal entropy cannot be homoclinically related.

\begin{theorem}\label{Thm:local-uniquess}
Let $X$ be a $C^{1+\beta}$ vector field ($\beta>0$) on a $C^\infty$ closed Riemannian manifold $M$, and assume $h_{\rm top}(\vf)>0$.
{If $\mu$ is an ergodic hyperbolic measure,
then there is at most one ergodic measure $\nu$ homoclinically related to $\mu$ such that }
$$
h_{\nu}(\varphi)=\sup\{h_{\eta}(\varphi): \eta~\text{is homoclinically related to}~\mu\}.
$$
\end{theorem}

\begin{proof}
The set $\sing$ carries no entropy.
Hence, if $\eta$ is ergodic with $h_{\eta}(\vf)>0$, then $\eta(\sing)=0$.
Moreover, the Ruelle inequality implies that $\eta$ is hyperbolic.
The result then follows from Corollary \ref{cor.local-uniq}, taking 
$\psi\equiv 0$.
\end{proof}

\subsection{Pesin blocks}

In this section we show that, under the assumptions of
Theorem \ref{Thm:finiteness-C-infty}, $\vf$ has finitely many pairwise non-homoclinically related ergodic measures of maximal entropy.
The key point, {which is a version for flows of the strong positive recurrence property introduced in \cite{BCS-SPR},} is that there exists a Pesin block to which every ergodic measure of maximal entropy gives positive measure to it.

The notion of Pesin block for $\vf$ is essentially defined through the time-one map $\varphi:=\varphi^1$.
For $x\notin {\rm Sing}$, let $E^c(x)$ be the subspace generated by $X(x)$.
Fix $\chi>0$. 

\medskip
\noindent
{\sc Pesin block ${\rm PES}_{\chi,\ell}^\varepsilon$:} For $0<\varepsilon\ll\chi$ and $\ell>0$, we define the {\em Pesin block} ${\rm PES}_{\chi,\ell}^\varepsilon$
to be the set of all regular points $x$ possessing a $D\varphi$-invariant splitting $T_{y}M=E^u(y)\oplus E^s(y) \oplus E^c(y)$ along the discrete orbit $\{\varphi^n(x):n\in\Z\}$ satisfying the following properties:
\begin{enumerate}
\item[(PB1)] $\|D\varphi^{-n}|_{E^u(\varphi^k(x))}\|\leq \ell e^{|k|\varepsilon}e^{-n\chi}$ and $\|D\varphi^{n}|_{E^s(\varphi^k(x))}\|\leq \ell e^{|k|\varepsilon}e^{-n\chi}$ for all $n\geq 0$ and $k\in \mathbb{Z}$;
\item[(PB2)] $\ell^{-1}e^{-(|k|+|n|)\varepsilon}\le\|D\varphi^{n}|_{E^c(\varphi^k(x))}\|\leq \ell e^{(|k|+|n|)\varepsilon}$ for all $n,k\in\mathbb Z$. 
\end{enumerate}

\medskip
By definition, ${\rm PES}_{\chi,\ell}^\varepsilon$ is compact.
The main results of this section are the following propositions.
{The first one is general, while the second requires the assumptions of 
Theorem \ref{Thm:finiteness-C-infty}.}

\begin{proposition}\label{Prop:HC-PES}
Let $X$ be a $C^{1+\beta}$ vector field ($\beta>0$) on a $C^\infty$ closed Riemannian manifold $M$.
If ${\rm PES}_{\chi,\ell}^{\varepsilon}$ is a Pesin block, then there are at most finitely many pairwise non-homoclinically related ergodic probability measures 
{$\mu$ with $\mu({\rm PES}_{\chi,\ell}^{\varepsilon})>0$.}
\end{proposition}
\begin{proposition}\label{Prop:SPR-FLow}
	Assume that $X$ is a $C^{\infty}$ vector field on the compact $3$-dimensional manifold $M$ and that $\{\varphi^t\}_{t\in \R}$ is the flow generated by $X$, with $h_{\rm top}(\varphi)>0$. 
	Then, there exists $\chi>0$ such that for every $0<\varepsilon \ll \chi$ and every $\tau>0$, there exist $\ell>0$ and $\delta\in(0,h_{\rm top}(\varphi))$ such that $\mu({\rm PES}_{\chi,\ell}^{\varepsilon})>1-\tau$ for every ergodic measure $\mu$ satisfying $h_{\mu}(\varphi)>h_{\rm top}(\varphi)-\delta$.
\end{proposition}

To prove Proposition \ref{Prop:HC-PES}, we recall some facts on the theory of local unstable manifolds on Pesin blocks. These can be derived from 
\cite[Section 7]{BP-SE}.

\begin{theorem}[Pesin]\label{Thm:stable-theorem}
Given $\chi>0$, for $0<\varepsilon \ll\chi$ and $\ell\in\mathbb N$, there is $\delta=\delta(\chi,\varepsilon,\ell)>0$ such that every $x\in {\rm PES}_{\chi,\ell}^{\varepsilon}$ has a local unstable manifold $W^u_\delta(x)$ of the form 
$$W^u_{\delta}(x)=\exp{x} \big\{(v,\psi^u_x(v)): v\in E^u(x), \|v\| \leq \delta \big\}$$
such that:
\begin{enumerate}[{\rm (1)}]
\item[{\rm (1)}] $\psi^u_x: \{v\in E^{u}(x):\|v\|\leq  \delta\} \to E^{c}(x) \oplus E^s(x)$ is a $C^{1+\beta}$ map with $\psi^u_x(0)=0$, $(D\psi^u_x)_0=0$ and ${\rm Lip}(\psi^u_x)\leq \tfrac{1}{10}$.
\item[{\rm (2)}]
For all $y,z\in W^u_{\delta}(x)$ it holds
$$d(\varphi^{-t}(y),\varphi^{-t}(z))\le 2\ell e^{-t(\chi-2\varepsilon)}d(y,z),\ \ \forall t\ge 0.$$
\end{enumerate}
An analogous statement holds for local stable manifolds $W^s_\delta(x)$. Additionally, the maps
$x\in {\rm PES}_{\chi,\ell}^{\varepsilon}\mapsto W^{u/s}_{\delta}(x)$ depend continuously in the $C^1$ topology.
\end{theorem}

A more general unstable manifold theorem for flows is stated in \cite[Section 7.3.5]{BP-SE}, and the proof of Theorem \ref{Thm:stable-theorem} follows by applying \cite[Theorem 7.1]{BP-SE} to the time-one map $\vf^1$.
Note that condition (PB2) in the definition of Pesin block implies the $\varepsilon$-slowly varying condition for the distance from the orbit to $\sing$.
Hence, local unstable manifolds for points of 
${\rm PES}_{\chi,\ell}^{\varepsilon}$ have uniformly positive size;  see \cite[Section 7.3.1]{BP-SE} for details.

For $x\in {\rm PES}_{\chi,\ell}^{\varepsilon}$, we define the global unstable and unstable manifolds of $x$ by
$$W^u(x):=\bigcup_{t> 0} \varphi^{t}(W^u_{\delta}(\varphi^{-t}(x)))\ \text{ and }\ W^s(x):=\bigcup_{t> 0} \varphi^{-t} (W^s_{\delta}(\varphi^{t}(x))).$$
Also, let $W^{cu/cs}_{\delta}(x):=\varphi^{[-\delta,\delta]}(W^{u/s}_{\delta}(x))$. 
By the continuity of $W^{u/s}_{\delta}(x)$ and the control of the flow direction on the Pesin block, we obtain the following local product structure.
\begin{lemma}\label{Lem:local-Pro}
Given $\chi>0$, let $0<\varepsilon \ll \chi$, $\ell \in \N$, and let $\delta=\delta(\chi,\varepsilon,\ell)>0$ be as in Theorem \ref{Thm:stable-theorem}. There exists $r=r(\chi,\varepsilon,\ell)>0$ such that
$$
\left(\begin{array}{c}
x,y\in {\rm PES}_{\chi,\ell}^{\varepsilon}\\
d(x,y)<r 
\end{array}
\right)
\ 
\Longrightarrow\ W^{cu}_{\delta}(x) \pitchfork W^{cs}_{\delta}(y)\neq \emptyset.
$$
\end{lemma} 

The condition is symmetric, hence we also have $W^{cs}_{\delta}(x) \pitchfork W^{cu}_{\delta}(y)\neq \emptyset$.

\begin{proof}[Proof of Proposition \ref{Prop:HC-PES}]
Let $r$ be as in Lemma \ref{Lem:local-Pro}.
Recalling that ${\rm PES}_{\chi,\ell}^{\varepsilon}$ is compact, we
choose finitely many points $\{z_i\}_{i=1}^{N}$ such that ${\rm PES}_{\chi,\ell}^{\varepsilon}\subset \bigcup_{i=1}^{n} B(z_i,r/2)$.
If two ergodic hyperbolic measures $\mu$ and $\nu$ are not homoclinically related,
then by Lemma \ref{Lem:local-Pro} there is no $z_i$ such that $\mu(B(z_i,r/2)\cap {\rm PES}_{\chi,\ell}^{\varepsilon})>0$ and $\nu(B(z_i,r/2)\cap {\rm PES}_{\chi,\ell}^{\varepsilon})>0$.
Therefore there are at most $N$ pairwise non-homoclinically related ergodic measures $\mu$ with $\mu({\rm PES}_{\chi,\ell}^{\varepsilon})>0$. 
\end{proof}

In the proof of Proposition \ref{Prop:SPR-FLow}, we will use the following version of Pliss' lemma, whose proof can be found in \cite[Appendix B]{CMY-22}.
\begin{lemma}\label{Lem:Pliss-Like}
	For every $0<\gamma_1<\gamma_2<C$ and every $\tau>0$, there is $\rho:=\rho(\gamma_1,\gamma_2,C,\tau)>0$ with the following property: if $\{a_n\}_{n\geq 0} \subset \R$ satisfies $|a_n|<C$ and 
	$$\liminf_{n\to \infty} \frac{1}{n} \# \{0\leq k<n:a_k<\gamma_1\}>1-\rho,$$
	then
    $$\limsup\limits_{n\to \infty} \frac{1}{n} \# \left\{0\leq k<n:\sum_{j=0}^{m-1} a_{j+k}\leq m \gamma_2,~\forall m>0\right\}>1-\tau.
    $$
\end{lemma}

\begin{proof}[Proof of Proposition \ref{Prop:SPR-FLow}]
The proof follows the approach developed in~\cite{BCS-SPR,BCS-CONTINUITY}.
Fix $\chi=h_{\rm top}(\varphi)/3$, $0<\varepsilon \ll \chi$ and $\tau>0$.
It suffices to show that for every sequence of ergodic measures $\{\nu_i\}_{i>0}$ satisfying $h_{\nu_i}(\varphi)\to h_{\rm top}(\varphi)$ and $\nu_i\to \mu$ for some measure $\mu$, there is $\ell>0$ such that $\nu_i({\rm PES}_{\chi,\ell}^{\varepsilon})>1-\tau$ for all sufficiently large $i$. 
Let $\gamma_1=\varepsilon/2$,  $\gamma_2=\varepsilon$, $C=\log \max\{\|D\varphi^1\|_{C^0},\|D\varphi^{-1}\|_{C^0}\}$ and 
$\tau'=(200C)^{-1}\varepsilon \tau$.
    
By the Ruelle inequality, $\nu_i$ has exactly one positive and one negative Lyapunov exponent.
By \cite[Theorem C]{BLY-ESPR}, there are $\ell_1,N_1\geq 1$ such that
\begin{equation}\label{eq:su-direction}
		\nu_i\left(\left\{x\in M:\begin{array}{l}\|D\varphi^{-n}|_{E^u(x)}\|\leq \ell_1 e^{-na\chi}\\
        \|D\varphi^{n}|_{E^s(x)}\|\leq \ell_1 e^{-n\chi}
        \end{array},\, \forall n>0\right\}\right)>1-\tau',\ \ \forall i>N_1. 
\end{equation}	
	
We now focus on estimating the behavior along the center direction.
By \cite{Newhouse-1989}, the entropy map is upper semi-continuous, hence $\mu$ is a measure of maximal entropy. In particular, $\mu$ is $\chi$-hyperbolic 
with $\mu(\sing)=0$.
Let $\rho:=\rho(\gamma_1,\gamma_2,C,\tau')>0$ given by Lemma \ref{Lem:Pliss-Like},
and take $L>0$ such that 
$$\mu\left(\left\{x\in M:e^{-\varepsilon L/2}<\|D\varphi^L|_{E^c(x)}\|<e^{\varepsilon L/2} \right\} \right)>1-\frac{\rho}{2}.$$
Since this latter set is open, $\exists N_2>N_1$ such that
$$\nu_i\left(\left\{x\in M:e^{-\varepsilon L/2}<\|D\varphi^L|_{E^c(x)}\|<e^{\varepsilon L/2} \right\}\right)>1-\rho,\ \ \forall i>N_2.$$
For each $x$, let $a_n(x):=\tfrac{1}{L}\log \|D\varphi^L|_{E^c(\varphi^{n-1}(x))}\|$.
By the Birkhoff ergodic theorem, for every $i>N_2$ there exists a set $\Lambda_i$  with $\nu_i(\Lambda_i)=1$ such that
$$\lim_{n\to \infty} \tfrac{1}{n} \# \left\{0\leq k<n: |a_k(x)|<\tfrac{\varepsilon}{2}\right\}>1-\rho,\ \ \forall x\in\Lambda_i.$$
Applying Lemma \ref{Lem:Pliss-Like}, for every $x\in \Lambda_i$ it holds
$$\lim_{n\to \infty} \frac{1}{n} \# \bigg\{0\leq k<n: -m\varepsilon<\sum_{j=0}^{m-1}a_{j+k}(x)<m\varepsilon,~\forall m>0\bigg\}>1-\tau'.$$
This implies that 
$$\nu_i\Big( \Big\{x\in M: e^{-mL\varepsilon}<\prod_{j=0}^{m-1} \|D\varphi^L|_{E^c(\varphi^j(x))}\| <e^{mL\varepsilon},~\forall m>0 \Big\} \Big)>
{1-\tau'}.$$
{Since $E^c$ has dimension one}, $\exists \ell_2:=\ell(L,C)$ such that
$$\ell_2^{-1} \cdot \|D\varphi^{m}|_{E^c(x)}\|^{L}\leq \prod_{j=0}^{m-1} \|D\varphi^L|_{E^c(\varphi^j(x))}\| \leq \ell_2 \cdot  \|D\varphi^{m}|_{E^c(x)}\|^{L},\ \ \forall m>0,$$
hence 
\begin{equation}\label{eq:c-direction}
		\nu_i\Big( \Big\{x\in M: \ell_2^{-1}  e^{-m\varepsilon}<\|D\varphi^m|_{E^c(x)}\| <\ell_2 e^{m\varepsilon},~\forall m>0 \Big\} \Big)>{1-\tau'}.
\end{equation}
Let $\ell=\max\{\ell_1,\ell_2\}$. 
Estimates \eqref{eq:su-direction}, \eqref{eq:c-direction} and \cite[Lemma 2.20]{BCS-SPR} then imply that for all $i>N_2$ it holds
$$\nu_i({\rm PES}_{\chi,\ell}^{\varepsilon})>1-\tfrac{100C}{\varepsilon} \tau'=1-\tfrac{\tau}{2}\cdot$$
This completes the proof of Proposition \ref{Prop:SPR-FLow}.
\end{proof}

\begin{proof}[Proof of Theorem \ref{Thm:finiteness-C-infty}]
By Proposition \ref{Prop:SPR-FLow}, there exists a Pesin block ${\rm PES}_{\chi,\ell}^{\varepsilon}$ such that $\mu({\rm PES}_{\chi,\ell}^{\varepsilon})>0.5$
for every ergodic measure  $\mu$ of maximal entropy.
By Proposition \ref{Prop:HC-PES}, there are finitely many pairwise non-homoclinically related ergodic measures of maximal entropy. 
By Theorem \ref{Thm:local-uniquess}, distinct ergodic measures of maximal entropy are not homoclinically related.
Hence $\vf$ has finitely many ergodic measures of maximal entropy.
\end{proof}

\medskip
\noindent
{\em{Acknowledgements.}} YL was supported by 
Instituto Serrapilheira grant
``Jangada Din\^{a}mica: Impulsionando Sistemas Din\^{a}micos na Regi\~{a}o Nordeste'', and
FAPESP grant number 2025/11400-7. DY was partially supported by National Key R\&D Program of China (2022YFA1005801), NSFC (12325106 \& 12526207), ZXL2024386 and Jiangsu Specially Appointed Professorship. CL is partially supported by NSFC 12501244.

\bigskip
\small
\bibliographystyle{plain-like-initial}
\bibliography{bib}

\bigskip
\bigskip
\bigskip

\hspace{-2.8cm}
\begin{tabular}{l l l l l}
\emph{J\'er\^ome Buzzi}
& &\emph{Sylvain Crovisier}
& &\emph{Yuri Lima}\\
Laboratoire de Math\'ematiques
&& Laboratoire de Math\'ematiques
&& Instituto de Matemática e Estatística\\
 d'Orsay, CNRS - UMR 8628
&&  d'Orsay, CNRS - UMR 8628
&&  Universidade de São Paulo (USP)\\
Universit\'e Paris-Saclay
&&  Universit\'e Paris-Saclay
&& Rua do Matão, 1010, Cid. Universitária\\
Orsay 91405, France
&& Orsay 91405, France
&& São Paulo -- SP, 05508-090, Brasil\\
\tt{jerome.buzzi}
&& \tt{sylvain.crovisier}
&&\tt{yurilima@gmail.com}\\
\; \tt{@universite-paris-saclay.fr}
&& \; \tt{@universite-paris-saclay.fr}
&&
\end{tabular}

\vspace{.5cm}
\begin{tabular}{l l l }
\emph{Chiyi Luo} &&\emph{Dawei Yang}\\
School of Mathematics and Statistics,
& & School of Mathematical Sciences\\
Jiangxi Normal University
& & Soochow University\\
Nanchang 330022, PR. China
& &
Suzhou 215006, PR. China\\
\tt{luochiyi98@gmail.com} && 
\tt{yangdw@suda.edu.cn} \\
\end{tabular}

\end{document}